\documentclass[12pt,a4paper]{article}

\usepackage{amsfonts}
\usepackage{amssymb,amsmath,amsthm}
\usepackage[utf8]{inputenc}
\usepackage[T1]{fontenc}
\usepackage{enumerate}
\usepackage{hyperref}
\usepackage{tikz}
\usepackage{tikz-3dplot}

\usetikzlibrary{calc}

\hypersetup{colorlinks=true,linkcolor=red, anchorcolor=green,
citecolor=cyan, urlcolor=red, filecolor=magenta, pdftoolbar=true}

\newtheorem{theorem}{Theorem}[section]
\newtheorem{lemma}[theorem]{Lemma}

\newtheorem{corollary}[theorem]{Corollary}
\newtheorem{remark}[theorem]{Remark}
\newtheorem{quest}[theorem]{Open Question}

\newcommand{\cemph}[1]{\emph{\color{red}#1}}

\newcommand\R{\mathbb{R}}
\newcommand\B{\mathbb{B}}
\newcommand\BM{\mathcal{BM}}
\newcommand\dbm{d_{BM}}

\DeclareMathOperator{\conv}{conv}
\DeclareMathOperator{\lin}{lin}
\DeclareMathOperator{\bd}{bd}

\newcommand\extrafootertext[1]{%
    \bgroup
    \renewcommand\thefootnote{\fnsymbol{footnote}}%
    \renewcommand\thempfootnote{\fnsymbol{mpfootnote}}%
    \footnotetext[0]{#1}%
    \egroup
}
\providecommand{\keywords}[1]{%
\extrafootertext{\textit{Key words and phrases.} #1.}
}
\providecommand{\subjclass}[1]{%
\extrafootertext{2020 \textit{Mathematics Subject Classification.} #1.}
}

\begin{document}

\title{Exact Banach--Mazur distances of certain $\ell_p$-sums and cones}
\author{Florian Grundbacher \MakeLowercase{and} Tomasz Kobos}

\maketitle

\begin{abstract}
We determine certain Banach--Mazur distances involving $\ell_p$-direct sums of finite-dimensional real normed spaces and related cone constructions of convex bodies. Using a recent characterization of the optimal Banach--Mazur position with respect to the Euclidean ball, we derive a closed formula for the distance from $X_1 \oplus_p \cdots \oplus_p X_k$ to Euclidean space in terms of the distances of the spaces $X_i$ to Euclidean space. For $p = 1$ we show that if $d_{BM}(X,\ell_1^n) \leq 3$, then $d_{BM}(X \oplus_1 \ell_1^m, \ell_1^{n+m}) = d_{BM}(X,\ell_1^n)$. Interpreting $\ell_1$-sums geometrically as double cones motivates a study of single cones over arbitrary convex bases, for which we establish an analogous result with the simplex replacing $\ell_1$. We further show that in dimension $3$ the distance between single cones with symmetric bases equals the distance between the bases, and that the same equality holds for double cones over planar symmetric bases in arbitrary dimension, under an additional assumption on the distance of the bases to $\ell_1^2$. As consequences, we obtain an explicit isometric embedding of the $2$-dimensional symmetric Banach--Mazur compactum into the $3$-dimensional (non-symmetric) compactum and lift a recent construction of arbitrarily large equilateral sets in the $2$-dimensional symmetric compactum to all higher dimensions.
\end{abstract}

\maketitle

\section{Introduction and Results}

The purpose of this paper is to determine exact Banach--Mazur distances between several $\ell_p$-sum constructions for finite-dimensional real normed spaces,
as well as related cone constructions for general convex bodies in $\R^n$.

Before we state the main results, 
we introduce the required notation and recall some basic facts.
For any integer $n \geq 1$ and scalar $p \in [1,\infty]$,
we denote by $\ell_p^n$ the space $\R^n$ equipped with the norm $\|\cdot\|_p$ defined for $x = (x_1, \ldots, x_n) \in \R^n$ by
\[
    \|x\|_p
    = ( |x_1|^p + \ldots + |x_n|^p )^{\frac{1}{p}},
\]
where for $p=\infty$ we understand $\|x\|_\infty = \max \{ |x_1|, \ldots, |x_n| \}$.
In particular, $\|\cdot\|_1$, $\|\cdot\|_2$, and $\| \cdot \|_{\infty}$ denote the $\ell_1$, Euclidean, and maximum norms, respectively.
The unit ball of $\ell_p^n$ is given by $\B_p^n$.
More generally, given $p \in [1, \infty]$ and finite-dimensional normed spaces $X_i=(\R^{n_i}, \| \cdot \|_{X_i})$ for $i=1, \ldots, k$, the \cemph{$\ell_p$-direct sum} $X_1 \oplus_p \ldots \oplus_p X_k$ is the space $\R^{n_1 + \ldots + n_k}$ equipped with the norm
\[
    \|x\|
    = \left( \|x^1\|^p_{X_1} + \ldots + \|x^k\|^p_{X_k} \right)^{\frac{1}{p}},
\]
where $x=(x^1, \ldots, x^k)$ with $x^i \in \R^{n_i}$ for each $i=1, \ldots, k$.
For $p = \infty$, we again understand $\|x\| = \max \{ \|x^1\|_{X_1}, \ldots, \|x^k\|_{X_k} \}$.

For a set $A \subseteq \R^n$, its \cemph{$v$-translation} by a vector $v \in \R^n$ is given by $A + v = \{ a + v : a \in A \}$.
Its \cemph{$\rho$-dilation} for $\rho \in \R$ is $\rho A = \{ \rho a : a \in A \}$.
We say that $A$ is \cemph{$0$-symmetric} if $(-1)A = A$, and \cemph{(centrally) symmetric} if $A+v$ is $0$-symmetric for some $v \in \R^n$.

The \cemph{Banach--Mazur distance} between two real normed spaces $X, Y$ of the same dimension is defined as
\[
    \dbm(X, Y)
    = \inf \|T\| \, \|T^{-1}\|,
\]
where the infimum (in fact a minimum in the finite-dimensional setting)
is taken over all invertible linear operators $T: X \to Y$
and the standard operator norms are used.
For \cemph{convex bodies} $K, L \subseteq \R^n$, i.e., compact convex sets with non-empty interior,
their \cemph{Banach--Mazur distance} is defined as
\[
    \dbm(K,L)
    = \inf \{\rho \geq 0 : K+u \subseteq T(L+v) \subseteq \rho (K+u) \},
\] 
with the infimum (again a minimum) taken over all invertible linear operators $T : \R^n \to \R^n$ and vectors $u, v \in \R^n$.
More generally, if $K$ and $L$ are \cemph{$d$-dimensional} convex bodies, meaning their affine spans have dimension $d$,
then $\dbm(K,L)$ is understood like above after embedding $K$ and $L$ in $\R^d$.
When $K$ and $L$ are $0$-symmetric,
the translation vectors $u$ and $v$ may be omitted.
Under the well-known correspondence of norms and $0$-symmetric convex bodies,
the equality $\dbm(X,Y) = \dbm(B_x,B_y)$ holds for any two normed spaces $X = (\R^n,\|\cdot\|_X)$, $Y = (\R^n,\|\cdot\|_Y)$ and their respective unit balls $B_X, B_Y \subseteq \R^n$.
The $n$-dimensional \cemph{Banach--Mazur compactum} $\BM^n$ is the set of affine equivalence classes of $n$-dimensional convex bodies equipped with the (multiplicative) metric $\dbm$.
By $\BM_s^n$ we denote the subset of $\BM^n$ corresponding to only $0$-symmetric convex bodies,
i.e., the $n$-dimensional symmetric Banach--Mazur compactum.

The $\ell_p$-direct sum of normed spaces is one of the fundamental constructions in the geometry of Banach spaces,
and many of their aspects have been studied in various different contexts.
In particular, estimates of the Banach--Mazur distance when $\ell_p$-sums are involved often play an important role, see, e.g., \cite{bourgainszarek, bourgain, johnson, johnson2, khrabrov, khrabrov2, plebanek, piasecki, rudelson, tomczaksymm}.

In the vast majority of these works, the Banach--Mazur distance to $\ell_p$-sums is only estimated rather than determined exactly.
This reflects a broader fact about Banach--Mazur compacta:
most of their properties have been established only in the asymptotic regime, i.e., as the dimension tends to infinity.
In part, this is due to exact computations of distances requiring solutions to optimization problems that can be challenging even in very specific situations.
One notable example is given by the classical spaces
$\ell_1^n$ and $\ell_{\infty}^n$.
Their distance is known to be of order $\sqrt{n}$,
but besides the trivial cases $n = 1,2$, the exact value has been determined only for $n=3$ as $\frac{9}{5}$ and $n=4$ as $2$ (see \cite{kobosvarivoda}).
The distance $\frac{9}{5}$ appears to be the largest Banach--Mazur distance computed for $3$-dimensional real normed spaces to date.

The main goal of this paper is to investigate some situations in which the Banach--Mazur distance to an $\ell_p$-sum can be determined exactly,
which leads us to explicit computations of the Banach--Mazur distance between many different pairs of normed spaces or convex bodies.
In concrete cases, obtaining an upper bound on the Banach--Mazur distance is usually straightforward,
whereas establishing the matching lower bound is the main challenge, requiring appropriate tools and often restrictions.
One of the few instances in which the Banach--Mazur distance to an $\ell_p$-sum has been computed precisely is \cite[Proposition~$2$]{tomczaksymm} by Tomczak-Jaegermann.
For any integer $n \geq 1$ and any $1 < p \leq 2$, it is proved that
\begin{equation}
\label{eq:ntj}
    \dbm(\ell_p^n \oplus_{p^*} \ldots \oplus_{p^*} \ell_p^n, \ell_2^{n^2})
    = n^{\frac{2-p}{p}},
\end{equation}
where $p^*=\frac{p}{p-1} \in [2, \infty)$ is the conjugate exponent of $p$ and there are exactly $n$ identical summands.
The proof method heavily relies on the large set of symmetries present in this particular sum and thus has limited applicability.

We begin with a broad generalization of \eqref{eq:ntj}, computing the Banach--Mazur distance between the Euclidean space and an arbitrary $\ell_p$-sum of finite-dimensional real normed spaces in terms of the distances of the individual summands.
The result reads as follows.

\begin{theorem}
\label{thm:l2_sum}
Let $n_1, n_2, \ldots, n_k \geq 1$ be integers and let $X_i$ be an $n_i$-dimensional real normed space for $i=1, \ldots, k$.
Then, for any $p \in [1,\infty]$,
\[
    \dbm(X_1 \oplus_p \ldots \oplus_p X_k, \ell^{n_1 + \ldots + n_k}_2)
    = \| (\dbm(X_1,\ell_2^{n_1}), \ldots, \dbm(X_k,\ell_2^{n_k}) \|_r,
\]
where $r = \frac{2p}{|p-2|} \in [2,\infty]$ (with the conventions $r = \infty$ when $p=2$ and $r=2$ when $p=\infty$).
\end{theorem}

The proof is presented in Section~\ref{sec:euclidean} and relies on a recent characterization \cite{grundbacherkobos} of the optimal Banach--Mazur position with respect to the Euclidean ball by an algebraic decomposition on contact points (see Theorem~\ref{thm:ader_cond} below),
inspired by an old paper of Ader \cite{ader}.
A key observation is that the $\ell_p$-sum operation preserves the optimal Banach--Mazur position when each summand is already optimally positioned relative to Euclidean space.
Some further consequences of Theorem~\ref{thm:l2_sum} are discussed in Section~\ref{sec:euclidean}.

In Section~\ref{sec:cones_to_fixed}, we turn our attention to $\ell_1$-sums with $\ell_1^m$ and related cone constructions.
Geometrically, if $X = (\R^n, \|\cdot\|)$ is a normed space with unit ball $B_X \subseteq \R^n$,
then the unit ball of $X \oplus_1 \ell_1^m$ in $\R^{n+m}$ is a double cone with $m$ symmetric vertex pairs over the base $B_X$.
This operation turns out to preserve the Banach--Mazur distance to $\ell_1$ space, provided that the initial distance is not too large.
By duality, an analogous result holds for $\ell_\infty$-sums with $\ell_\infty^m$,
which geometrically corresponds to forming a cylinder over the unit ball of $X$.

\begin{theorem}
\label{thm:l1_sum}
Let $n, m \geq 1$ be integers and let $X$ be an $n$-dimensional real normed space with $\dbm(X, \ell_1^n) \leq 3$. Then
\[
    \dbm(X \oplus_1 \ell^m_1, \ell^{n+m}_1)
    = \dbm(X, \ell_1^n).
\]
Similarly, if $\dbm(X, \ell_{\infty}^n) \leq 3$, then
\[
    \dbm(X \oplus_{\infty} \ell^m_{\infty}, \ell^{n+m}_{\infty})
    = \dbm(X, \ell_{\infty}^n).
\]
\end{theorem}

With the diameter of $\BM^n_s$ being at most $n$, we obtain the following immediate corollary.

\begin{corollary}
For integers $n, m \geq 1$ and a real normed space $X$ with $\dim(X) = n \leq 3$,
\[
    \dbm(X \oplus_1 \ell^m_1, \ell^{n+m}_1)
    = \dbm(X, \ell_1^n)
        \quad \text{ and } \quad
    \dbm(X \oplus_{\infty} \ell^m_{\infty}, \ell^{n+m}_{\infty})
    = \dbm(X, \ell_{\infty}^n).
\]
\end{corollary}

Since the $\ell_1$-sum corresponds to forming a centrally symmetric double cone over a symmetric base,
it is natural to consider also (non-symmetric) single cones over arbitrary convex bases.
In this setting, we obtain the following analog of Theorem~\ref{thm:l1_sum},
where the cross-polytope (i.e., the unit ball of $\ell_1$) is replaced by the simplex.
By $\Delta^n \subseteq \R^n$ we mean any \cemph{simplex} in $\R^n$,
and by $\conv(A)$ the \cemph{convex hull} of a set $A \subseteq \R^n$.

\begin{theorem}
\label{thm:simplex}
Let $n, m \geq 1$ be integers and let $B \subseteq \R^{n+m}$ be an $n$-dimensional convex body with $\dbm(B, \Delta^n) \leq 2$.
Let $v^1, \ldots, v^m \in \R^{n+m}$ be arbitrary vectors such that $C = \conv( B \cup \{v^1, \ldots, v^m \} ) \subseteq \R^{n+m}$ is an $(n+m)$-dimensional convex body.
Then
\[
    \dbm(C, \Delta^{n+m})
    = \dbm(B, \Delta^n).
\]
\end{theorem}

The key idea for establishing the matching lower bound for the Banach--Mazur distance in Theorems~\ref{thm:l1_sum}~and~\ref{thm:simplex} is the same in both cases:
for any Banach--Mazur position of the relevant convex bodies, there exists a linear projection onto a lower-dimensional subspace that maps the (double) cones to convex bodies that are affinely equivalent to the respective cones' bases.

This projective approach is further developed in Section~\ref{sec:between_cones}, where we first consider single cones in $\R^3$ over planar symmetric bases.
In this setting, the distance between the cones turns out to coincide with the distance between their bases.

\begin{theorem}
\label{thm:3d_cones}
Let $B_1, B_2 \subseteq \R^3$ be $2$-dimensional symmetric convex bodies
and let $v^1, v^2 \in \R^3$ be arbitrary vectors such that $C_i = \conv(B_i \cup \{v^i\})$ is a $3$-dimensional convex body for $i=1, 2$.
Then
\begin{equation}
\label{eq:3d_cones_dist}
    \dbm(C_1, C_2)
    = \dbm(B_1, B_2).
\end{equation}
\end{theorem}

As a corollary, we obtain an explicit isometric embedding of the $2$-dimensional symmetric Banach--Mazur compactum $\BM^2_s$ into the $3$-dimensional (non-symmetric) Banach--Mazur compactum $\BM^3$.
To our knowledge, this is surprisingly the first instance of such an embedding between Banach--Mazur compacta of different dimensions (apart from the trivial embedding of the $1$-dimensional compactum).
Here and throughout the paper, by $e^1, \ldots, e^n$ we denote the \cemph{canonical unit basis} of $\R^n$.

\begin{corollary}
\label{cor:3d_embedd}
The mapping $\varphi \colon \BM^2_s \to \BM^3$ defined by
\[
    \varphi(C)
    = \conv \left( (C \times \{0\}) \cup \{\pm e^3\} \right)
    \quad \text{for } C \in \BM^2_s
\]
is an isometric embedding of $\left( \BM^2_s, \dbm \right)$ into $\left( \BM^3, \dbm \right)$.
\end{corollary}

It is natural to conjecture that an equality analogous to \eqref{eq:3d_cones_dist} holds for centrally symmetric double cones, which would strengthen Corollary~\ref{cor:3d_embedd} to an embedding of $\BM_s^2$ into $\BM_s^3$.
We do not know whether this is true in general.
However, we can prove it under the additional assumption that at least one base is not too close to $\ell_1^2$, in exchange for embedding the double cones into arbitrary dimensions.
The necessity of this assumption is discussed after the proof of the theorem in Section~\ref{sec:between_cones} (see also Open Question~\ref{que:l1} and Remark~\ref{rem:l1} for related details).

\begin{theorem}
\label{thm:sym_cones}
Let $X_1$ and $X_2$ be $2$-dimensional real normed spaces
with $\dbm(X_1, \ell_1^2) \geq \dbm(X_1,X_2)$ or $\dbm(X_2, \ell_1^2) \geq \dbm(X_1,X_2)$.
Then, for any integer $m \geq 1$,
\[
    \dbm(X_1 \oplus_1 \ell_1^m, X_2 \oplus_1 \ell_1^m)
    = \dbm(X_1,X_2).
\]
\end{theorem}

By duality, an analogous result holds for the $\ell_{\infty}$-sums with $\ell_{\infty}^m$.
As an application, we obtain a local isometric embedding of $\BM^2_s \setminus \{ \ell_1^2 \}$ into $\BM^n_s$ for any $n \geq 3$.
In the following corollary, we regard the symmetric Banach--Mazur compactum $\BM^n_s$ as consisting of isometry classes of $n$-dimensional normed spaces.

\begin{corollary}
\label{cor:sym_embedd}
For an integer $n \geq 3$, let $\psi: \BM^2_s \to \BM^n_s$ be the mapping defined by
\[
    \psi(X)
    = X \oplus_1 \ell_1^{n-2}
    \quad \text{for } X \in \BM^2_s.
\]
Suppose for $X \in \BM^2_s \setminus \{ \ell_1^2 \}$ that $X_1, X_2 \in \BM^2_s$ satisfy
$\dbm(X, X_i) \leq \dbm(X, \ell_1^2)^{\frac{1}{3}}$ for $i=1, 2$.
Then 
\[
    \dbm(\psi(X_1), \psi(X_2))
    = \dbm(X_1, X_2).
\]
In particular, $\psi$ is a local isometric embedding of $\left( \BM^2_s \setminus \{ \ell_1^2 \}, \dbm \right)$ into  $\left( \BM^n_s, \dbm \right)$.
\end{corollary}

As a further consequence, we lift a recent construction of equilateral sets in the $2$-dimensional symmetric Banach--Mazur compactum \cite{kobosswanepoel} to higher dimensions.
Recall that a set $S$ in a metric space is called \cemph{equilateral} if there exists $d > 0$ such that all distinct $x,y \in S$ are at distance $d$.

\begin{corollary}
\label{cor:eq_sets}
For integers $n, m \geq 2$, there exist equilateral sets of cardinality $m$ in $\BM_s^n$.
\end{corollary}

Proofs of Corollaries~\ref{cor:sym_embedd}~and~\ref{cor:eq_sets} are given at the end of Section~\ref{sec:between_cones}.

Throughout the paper, we frequently use the equality $\dbm(X,Y) = \dbm(B_X, B_Y)$ for normed spaces $X = (\R^n,\|\cdot\|_X)$, $Y = (\R^n,\|\cdot\|_Y)$ and their respective unit balls $B_X, B_Y \subseteq \R^n$, choosing whichever viewpoint is more convenient.
We also refer to the well-known duality of $\ell_p$-sums, namely
\[
    (X_1 \oplus_p \ldots \oplus_p X_k)^*
    = X^*_1 \oplus_{p^*} \ldots \oplus_{p^*} X^*_k
\]
for any finite-dimensional normed spaces $X_1, \ldots, X_k$ and $p \in [1, \infty]$,
where $X^*$ denotes the \cemph{dual space} of a normed space $X$ and $p^*=\frac{p}{p-1} \in [1,\infty]$ is the conjugate exponent of $p$.
Given a vector $x \in \R^n$, we write $x_1, \ldots, x_n$ for the coordinates of $x$ in the canonical unit basis of $\R^n$.

\section{The Distance of \texorpdfstring{$\ell_p$}{lp}-Sums to the Euclidean Space}
\label{sec:euclidean}

In this section, we prove Theorem~\ref{thm:l2_sum}.
The main ingredient is the following characterization of the optimal Banach--Mazur position with respect to the Euclidean ball.
Similar to the characterization of the maximal-volume (John) ellipsoid by the decomposition of the identity on the contact points,
the optimal Banach--Mazur distance ellipsoids
can be described by an algebraic condition on the contact points.
For the proof, historical context, and different applications of this result, the reader is referred to \cite{grundbacherkobos}.
By $\bd(A)$ we denote the \cemph{boundary} of a set $A \subseteq \R^n$.

\begin{theorem}
\label{thm:ader_cond}
Let $K \subseteq \R^n$ be a $0$-symmetric convex body and let $R \geq r > 0$ be such that $r \B^n \subseteq K \subseteq R \B^n$. Then the following are equivalent:
\begin{enumerate}[(i)]
\item $\dbm(K,\B^n)=\frac{R}{r}$.
\item There exist integers $N, M \geq 1$, inner contact points $y^1, \ldots, y^N \in \bd(K) \cap \bd(r\B^n)$, outer contact points $z^1, \ldots, z^M \in \bd(K) \cap \bd(R\B^n)$, as well as weights $\lambda_1, \ldots, \lambda_N$, $\mu_1, \ldots, \mu_M > 0$ such that
\[
    \sum_{i=1}^N \lambda_i y^i (y^i)^T
    = \sum_{j=1}^M \mu_j z^j (z^j)^T.
\]
\end{enumerate}
\end{theorem}

The majority of the work for proving Theorem~\ref{thm:l2_sum} is handled by the following lemma.
For $0$-symmetric convex bodies $C_i \subseteq \R^{n_i}$, where $i=1, \ldots, k$, their $\ell_p$-sum $C_1 \oplus_p \ldots \oplus_p C_k$ is the $0$-symmetric unit ball in $\R^{n_1 + \ldots + n_k}$ of the $\ell_p$-sum of the corresponding normed spaces with unit balls $C_i$.

\begin{lemma}
\label{lem:l2_sum}
Let $n_1, n_2, \ldots, n_k \geq 1$ be integers and let $C_i \subseteq \R^{n_i}$ be a $0$-symmetric convex body for $i=1, \ldots, k$.
Suppose that for $d_i = \dbm(C_i, \B_2^{n_i})$ the inclusions $\B_2^{n_i} \subseteq C_i \subseteq d_i \B_2^{n_i}$ hold for each $i=1, \ldots, k$.
Then, for every $p \in [2, \infty]$,
\begin{equation}
\label{eq:l2_sum_chain}
    \B_2^{n_1 + \ldots + n_k}
    \subseteq C_1 \oplus_p \ldots \oplus_p C_k
    \subseteq \|(d_1, \ldots, d_k)\|_r \B_2^{n_1 + \ldots + n_k}
\end{equation}
and
\[
    \dbm(C_1 \oplus_p \ldots \oplus_p C_k, \B_2^{n_1 + \ldots + n_k})
    = \|(d_1, \ldots, d_k)\|_r,
\]
where $r = \frac{2p}{p-2} \in [2,\infty]$ (with the conventions $r = \infty$ when $p = 2$ and $r = 2$ when $p = \infty$).
\end{lemma}
\begin{proof}
Let $\| \cdot\|_{C_i}$ denote the norm on $\R^{n_i}$ whose unit ball is $C_i$.
By assumption,
\begin{equation}
\label{eq:Ci_norm_ineqs}
    \|x^i\|_{C_i}
    \leq \|x^i\|_2
    \leq d_i \, \|x^i\|_{C_i}
\end{equation}
for all $x^i \in \R^{n_i}$.
Let $C = C_1 \oplus_p \ldots \oplus_p C_k$ and let $\| \cdot \|_C$ be the corresponding norm on $\R^{n_1 + \ldots + n_k}$.
By \eqref{eq:Ci_norm_ineqs} and $p \geq 2$, for $x = (x^1, \ldots, x^k) \in \R^{n_1 + \ldots + n_k}$ with $x^i \in \R^{n_i}$ for $i=1, \ldots, k$ we can estimate
\begin{equation}
\label{eq:l2_sum_lower_bound}
    \|x\|_C
    = \left( \sum_{i=1}^{k} \|x^i\|^p_{C_i} \right)^{\frac{1}{p}}
    \leq \left( \sum_{i=1}^{k} \|x^i\|^p_{2} \right)^{\frac{1}{p}}
    \leq \left( \sum_{i=1}^{k} \|x^i\|^2_{2} \right)^{\frac{1}{2}}
    = \|x\|_2.
\end{equation}
For the estimate in the other direction, \eqref{eq:Ci_norm_ineqs} and H\"{o}lder's inequality with exponents $\frac{p}{p-2}$ and $\frac{p}{2}$ (which extends naturally to the cases of $p=2$ and $p=\infty$) yield
\begin{equation}
\begin{split}
\label{eq:l2_sum_upper_bound}
    \|x\|^2_2
    & = \sum_{i=1}^{k} \|x^i\|^2_{2}
    \leq \sum_{i=1}^{k} d_i^2 \, \|x^i\|^2_{C_i}
    \leq \left( \sum_{i=1}^{k} d_i^{\frac{2p}{p-2}} \right)^{\frac{p-2}{p}} \left( \sum_{i=1}^{k} \|x^i\|^{p}_{C_i} \right)^{\frac{2}{p}}
    \\
    & = \|(d_1, \ldots, d_k)\|^2_r \, \|x\|_C^2.
\end{split}
\end{equation}
This establishes \eqref{eq:l2_sum_chain}.

It remains to show that this chain of inclusions is optimal for the Banach--Mazur distance.
By a simple inductive argument, it suffices to consider only the case of $k=2$.
Moreover, by a continuity argument, we can suppose that $p \in (2,\infty)$.
Our goal is to use the characterization of the Banach--Mazur position with respect to the Euclidean ball given in Theorem~\ref{thm:ader_cond}.

Since the positions $\B_2^{n_i} \subseteq C_i \subseteq d_i \B_2^{n_i}$ are optimal for the Banach--Mazur distance for $i=1, 2$, Theorem~\ref{thm:ader_cond} provides integers $N_1, N_2, M_1, M_2 \geq 1$,
four sets of contact points 
$y^1, \ldots, y^{N_1} \in \bd(C_1) \cap \bd(\B_2^{n_1})$,
$z^1, \ldots, z^{M_1} \in \bd(C_1) \cap \bd(d_1 \B_2^{n_1})$,
$u^1, \ldots, u^{N_2} \in \bd(C_2) \cap \bd(\B_2^{n_2})$, and
$v^1, \ldots, v^{M_2} \in \bd(C_2) \cap \bd(d_2 \B_2^{n_2})$,
as well as weights $\alpha_i, \beta_i, \gamma_i, \delta_i > 0$ such that
\begin{equation}
\label{eq:ader}
    \sum_{i=1}^{N_1} \alpha_i y^i (y^i)^T
    = \sum_{i=1}^{M_1} \beta_i z^i (z^i)^T
        \quad \text { and } \quad
    \sum_{i=1}^{N_2} \gamma_i u^i (u^i)^T
    = \sum_{i=1}^{M_2} \delta_i v^i (v^i)^T.
\end{equation}
By pairing each contact point with its negative and normalizing, we may further assume
\begin{equation}
\label{eq:normalized}
    \sum_{i=1}^{N_1}\alpha_i y^i
    = \sum_{i=1}^{M_1} \beta_i z^i
    = \sum_{i=1}^{N_2} \gamma_i u^i
    = \sum_{i=1}^{M_2} \delta_i v^i = 0
        \quad \text{ and } \quad
    d_2^{\frac{4}{p-2}} \sum_{i=1}^{M_1} \beta_i
    = d_1^{\frac{4}{p-2}} \sum_{i=1}^{M_2} \delta_i
    = 1.
\end{equation}
Our goal is to determine contact points of $C = C_1 \oplus_p C_2$ with $\B_2^{n_1+n_2}$ and $\|(d_1, d_2)\|_r \B_2^{n_1+n_2}$
for which an analogous algebraic decomposition exists.

We first note that for each $i=1, \ldots,  N_1$ the point $(y^i, 0) \in \R^{n_1+n_2}$ is a common point of $\bd(C)$ and $\bd(\B_2^{n_1+n_2})$
since such a point satisfies \eqref{eq:l2_sum_lower_bound} with equality.
Similarly, for each $i=1, \ldots, N_2$ the point $(0, u^i) \in \R^{n_1+n_2}$ is a common point of $\bd(C)$ and $\bd(\B_2^{n_1+n_2})$.
On the other hand, the contact points of $C$ and $\|(d_1, d_2)\|_r \B_2^{n_1+n_2}$ include all points of the form 
\[
    (\widetilde{z}^i, \widetilde{v}^j)
    = \frac{\|(d_1, d_2)\|_r}{\sqrt{ d_1^{\frac{2p}{p-2}} + d_2^{\frac{2p}{p-2}} }} \left( d_1^{\frac{2}{p-2}} z^i, d_2^{\frac{2}{p-2}} v^j \right)
    \in \R^{n_1+n_2}
\]
for $i = 1, \ldots, M_1$ and $j = 1, \ldots, M_2$.
Indeed, $\frac{1}{d_1} \|z^i\|_2 = \|z^i\|_{C_1} = 1$ and $\frac{1}{d_2} \|v^j\|_2 = \|v^j\|_{C_2} = 1$
yield that such points achieve equality in the estimate \eqref{eq:l2_sum_upper_bound} by
\begin{align*}
    \left\| \left( d_1^{\frac{2}{p-2}} z^i, d_2^{\frac{2}{p-2}} v^j \right) \right\|_2^2
    & = d_1^{\frac{2p}{p-2}} + d_2^{\frac{2p}{p-2}}
    = \left( d_1^{\frac{2p}{p-2}} + d_2^{\frac{2p}{p-2}} \right)^{\frac{p-2}{p}} \left( d_1^{\frac{2p}{p-2}} + d_2^{\frac{2p}{p-2}} \right)^{\frac{2}{p}}
    \\
    & =\left( d_1^{\frac{2p}{p-2}} + d_2^{\frac{2p}{p-2}} \right)^{\frac{p-2}{p}} \left( \left\|d_1^{\frac{2}{p-2}} z^i \right\|^{p}_{C_1} + \left\| d_2^{\frac{2}{p-2}} v^j \right\|^{p}_{C_2} \right)^{\frac{2}{p}}.
\end{align*}
To the inner contact points $(y^i, 0)$ we assign the weights $\eta_i := \alpha_i$,
to the inner contact points $(0, u^i)$ the weights $\kappa_i := \gamma_i$,
and to the outer contact points $(\widetilde{z}^i, \widetilde{v}^j)$
the weights 
\[
    \omega_{ij}
    := \frac{d_1^{\frac{2p}{p-2}} + d_2^{\frac{2p}{p-2}}}{\| (d_1, d_2) \|_r^2} \, \beta_i \delta_j.
\]
Then,
\[
    \eta_i (y^i, 0) (y^i, 0)^T
    = \alpha_i \left[ \begin{array}{cc}
        y^i (y^i)^T & 0
        \\
        0 & 0
    \end{array} \right]
        \quad \text{ and } \quad
    \kappa_i (0, u^i) (0, u^i)^T
    = \gamma_i \left[ \begin{array}{cc}
        0 & 0
        \\
        0 & u^i(u^i)^T 
    \end{array} \right].
\]
Moreover,
\[
    \omega_{ij} (\widetilde{z}^i, \widetilde{v}^j) (\widetilde{z}^i, \widetilde{v}^j)^T
    = \beta_i \delta_j \left[\begin{array}{cc}
        d_1^{\frac{4}{p-2}} z^i (z^i)^T & (d_1 d_2)^{\frac{2}{p-2}} z^i (v^j)^T
        \\
        (d_1 d_2)^{\frac{2}{p-2}} v^j (z^i)^T & d_2^{\frac{4}{p-2}} v^j (v^j)^T 
    \end{array} \right].
\]
Therefore,
\[
    \sum_{i=1}^{N_1} \eta_i (y^i, 0) (y^i, 0)^T + \sum_{j=1}^{N_2} \kappa_j (0, u^j) (0, u^j)^T
    = \left[\begin{array}{cc}
        \sum_{i=1}^{N_1} \alpha_i y^i (y^i)^T & 0
        \\
        0 & \sum_{j=1}^{N_2} \gamma_j u^j (u^j)^T 
    \end{array} \right],
\]
and similarly
\[
    \sum_{i=1}^{M_1} \sum_{j=1}^{M_2} \omega_{ij} (\widetilde{z}^i, \widetilde{v}^j) (\widetilde{z}^i, \widetilde{v}^j)^T
    = \sum_{i=1}^{M_1} \sum_{j=1}^{M_2} \beta_i \delta_j \left[\begin{array}{cc}
        d_1^{\frac{4}{p-2}} z^i (z^i)^T & (d_1 d_2)^{\frac{2}{p-2}} z^i (v^j)^T
        \\
        (d_1 d_2)^{\frac{2}{p-2}} v^j (z^i)^T & d_2^{\frac{4}{p-2}} v^j (v^j)^T 
    \end{array} \right].
\]
From \eqref{eq:normalized} it follows that
\[
    \sum_{i=1}^{M_1} \sum_{j=1}^{M_2} \beta_i \delta_j z^i (v^j)^T = \left( \sum_{i=1}^{M_1} \beta_i z^i \right) \left( \sum_{j=1}^{M_2} \delta_j v^j \right)^T
    = 0,
\]
and in the same way $\sum_{i=1}^{M_1} \sum_{j=1}^{M_2} \beta_i \delta_j v^j (z^i)^T = 0$.
Moreover, conditions \eqref{eq:ader} and \eqref{eq:normalized} yield
\[
    \sum_{i=1}^{M_1} \sum_{j=1}^{M_2} \beta_i \delta_j d_1^{\frac{4}{p-2}} z^i (z^i)^T
    = \left( d_1^{\frac{4}{p-2}} \sum_{j=1}^{M_2} \delta_j \right) \left( \sum_{i=1}^{M_1} \beta_i z^i (z^i)^T \right)
    = \sum_{i=1}^{N_1} \alpha_i y^i (y^i)^T
\]
and similarly $\sum_{i=1}^{M_1} \sum_{j=1}^{M_2} \beta_i \delta_j d_2^{\frac{4}{p-2}} v^j (v^j)^T =  \sum_{j=1}^{N_2} \gamma_j u^j (u^j)^T$.
This establishes that
\[
    \sum_{i=1}^{N_1} \eta_i (y^i, 0) (y^i, 0)^T + \sum_{j=1}^{N_2} \kappa_j (0, u^j) (0, u^j)^T
    = \sum_{i=1}^{M_1} \sum_{j=1}^{M_2} \omega_{ij} (\widetilde{z}^i, \widetilde{v}^j) (\widetilde{z}^i, \widetilde{v}^j)^T,
\]
which shows that the required decomposition on the contact points exists.
Theorem~\ref{thm:ader_cond} now yields that the chain of inclusions \eqref{eq:l2_sum_chain} is optimal for the Banach--Mazur distance between $C$ and $\B_2^{n_1+n_2}$, that is, $\dbm(C, \B_2^{n_1+n_2}) = \|(d_1, d_2)\|_r$.
This completes the proof.
\end{proof}

With the previous lemma established, the proof of Theorem~\ref{thm:l2_sum} is straightforward.

\begin{proof}[Proof of Theorem~\ref{thm:l2_sum}]
For $p \geq 2$, the assertion follows directly from Lemma~\ref{lem:l2_sum}.
Indeed, for each $i = 1, \ldots, k$ there exists a normed space $X'_i$
isometric to $X_i$ whose unit ball $C_i$ satisfies $\B_2^{n_i} \subseteq C_i \subseteq d_i \B_2^{n_i}$ for $d_i = \dbm(C_i, \B_2^{n_i})$.
Applying Lemma~\ref{lem:l2_sum} to these unit balls, we obtain that the sum $X'_1 \oplus_p \ldots \oplus_p X'_k$ has the required distance to the Euclidean space.
Since the spaces $X'_1 \oplus_p \ldots \oplus_p X'_k$ and $X_1 \oplus_p \ldots \oplus_p X_k$ are clearly isometric, their distances to the Euclidean space are the same.

The case $1 \leq p < 2$ follows easily by duality,
as for any finite-dimensional normed spaces $X, Y$ of the same dimension we have $\dbm(X^*, Y^*) = \dbm(X, Y)$.
In particular, $\dbm(X_i^*, \ell_2^{n_i}) = \dbm(X_i, \ell_2^{n_i})=d_i$ for each $i=1, \ldots k$.
If now $p^* = \frac{p}{p-1} \in (2, \infty]$
denotes the conjugate exponent of $p$, then the duality of $\ell_p$-sums and the first case together yield
\[
    \dbm(X_1 \oplus_p \ldots \oplus_p X_k, \ell^{n_1 + \ldots + n_k}_2)
    = \dbm(X_1^* \oplus_{p*} \ldots \oplus_{p*} X_k^*, \ell^{n_1 + \ldots + n_k}_2)
    = \|(d_1, \ldots, d_k)\|_{r^*},
\]
where
\[
    r^*
    = \frac{2 p^*}{p^*-2}
    = \frac{\frac{2p}{p-1}}{\frac{2-p}{p-1}}
    = \frac{2p}{2-p}
    = r,
\]
as required.
\end{proof}

Theorem~\ref{thm:l2_sum} gives a simple $1$-line proof of the well-known Banach--Mazur distance between $\ell_2^n$ and $\ell_p^n$, so it in particular generalizes \eqref{eq:ntj}, as pointed out in the introduction.
Another obvious consequence is the following.

\begin{corollary}
Let $n_1, \ldots, n_k \geq 1$ be integers and let $X_i$ be an $n_i$-dimensional real normed space with $d_i = \dbm(X_i, \ell_2^{n_i})$ for $i = 1, \ldots, k$.
Define
\[
    f(p)
    = \dbm(X_1 \oplus_p \ldots \oplus_p X_k, \ell_2^{n_1 + \ldots + n_k})
\]
for $p \in [1, \infty]$.
Then
\[
    \max\{d_1, \ldots, d_k \}
    = f(2)
    \leq f(p)
    = f \left( \frac{p}{p-1} \right)
    \leq f(1)
    = f(\infty)
    = \sqrt{d_1^2 + \ldots + d_k^2}.
\]
\end{corollary}

As a further application of Theorem~\ref{thm:l2_sum}, we obtain that every Hanner polytope achieves the maximal possible Banach--Mazur distance to the Euclidean ball among $n$-dimensional centrally symmetric convex bodies.
Recall that a \cemph{Hanner polytope} is a centrally symmetric polytope that arises by taking arbitrary iterated $\ell_1$- and $\ell_\infty$-sums of segments or other lower-dimensional Hanner polytopes.
These polytopes were introduced by Hanner \cite{hanner} and are particularly well-known for their role in the famous Mahler conjecture,
where they are conjectured to be the only minimizers of the Mahler volume product among centrally symmetric convex bodies.
It was proved by Kim \cite{kim} that every Hanner polytope is indeed a strict local minimizer of this product with respect to the Banach--Mazur distance.
With regard to the Banach--Mazur distance itself, it appears that specialists in the field have long been aware that Hanner polytopes maximize the distance to the Euclidean ball among centrally symmetric convex bodies, where the John Ellipsoid Theorem shows that this maximal distance is at most $\sqrt{n}$.
However, to our knowledge, no proof of this fact for general Hanner polytopes has appeared in the published literature.
The following corollary provides an immediate proof via Theorem~\ref{thm:l2_sum}.

\begin{corollary}
Every Hanner polytope in $\R^n$ has the Banach--Mazur distance $\sqrt{n}$ to the $n$-dimensional Euclidean ball.
\end{corollary}
\begin{proof}
We proceed by induction on the dimension $n \geq 1$. 
There is nothing to prove for $n=1$,
so let us assume $n \geq 2$.
Suppose that $C \subseteq \R^n$ is a Hanner polytope that is the unit ball of the normed space $X = (\R^n, \| \cdot \|_X)$.
By the definition of a Hanner polytope, there exist normed spaces $X_i = (\R^{n_i}, \| \cdot \|_{X_i})$ for $i = 1, 2$ with some Hanner polytopes as their unit balls such that $n_1+n_2=n$ and either $X = X_1 \oplus_1 X_2$ or $X = X_1 \oplus_{\infty} X_2$.
By the inductive assumption and Theorem~\ref{thm:l2_sum}, we have in both cases
that
\[
    \dbm(X, \ell_2^n)
    = \sqrt{\dbm(X_1, \ell_2^{n_1})^2 + \dbm(X_2, \ell_2^{n_2})^2}
    = \sqrt{n_1 + n_2}
    = \sqrt{n}.\qedhere
\]
\end{proof}

Let us point out that Hanner polytopes are not the only maximizers of the Banach--Mazur distance to the Euclidean ball among centrally symmetric convex bodies when $n \geq 4$,
even in the class of polytopes (see \cite[Example~$2.5$]{grundbacherkobos}).

We close this section with the following remark, in which we observe how
Lemma~\ref{lem:l2_sum} extends to the non-symmetric setting.
The $\ell_p$-sum $C_1 \oplus_p \cdots \oplus_p C_k$ of (not necessarily symmetric) convex bodies containing the origin in their interiors is defined analogously to before,
but using the gauge (or Minkowski) functional induced by each $C_i$ instead of a norm.

\begin{remark}
Lemma~\ref{lem:l2_sum} remains valid for general (not necessarily $0$-symmetric) convex bodies $C_i \subseteq \R^{n_i}$ satisfying $\B_2^{n_i} \subseteq C_i \subseteq d_i \B_2^{n_i}$,
where $d_i = \dbm(C_i, \B_2^{n_i})$.
In the non-symmetric case, the characterization of the optimal Banach--Mazur position with respect to the Euclidean ball (the analog of Theorem~\ref{thm:ader_cond})
requires an additional balancing condition (see \cite[Theorem~$1.5$]{grundbacherkobos2}), which in the considered case writes as
\[
    \sum_{i=1}^{N} \lambda_i y^i
    = \sum_{j=1}^{M} \mu_j z^j
    = 0.
\]
In the symmetric case, this vanishing condition can always be imposed by pairing each contact point with its negative and distributing the weights accordingly.
Consequently, in the non-symmetric case the equalities \eqref{eq:normalized} can still be assumed,
and one readily verifies that the analogous balancing condition is also satisfied by the constructed contact points of $C_1 \oplus_p C_2$ with $\B_2^{n_1+n_2}$ and $\|(d_1, d_2)\|_r \B_2^{n_1+n_2}$ in the proof of the lemma.
\end{remark}

\section{The Distance of (Double) Cones to \texorpdfstring{$\ell_1^n$}{l1n} and the Simplex}
\label{sec:cones_to_fixed}

In this section, we prove Theorems~\ref{thm:l1_sum}~and~\ref{thm:simplex}.
While obtaining the exact values of the Banach--Mazur distances as upper bounds is straightforward in both cases,
verifying the matching lower bounds is more involved and relies on appropriately projecting everything onto the bases of the cones.
Clearly, for convex bodies $C_1 \subseteq C_2 \subseteq dC_1 \subseteq \R^n$ and a linear projection $P$,
the inclusions are preserved to $P(C_1) \subseteq P(C_2) \subseteq dP(C_1)$.
Thus, whenever the images $P(C_1)$ and $P(C_2)$ can be controlled, obtaining lower bounds on $\dbm(C_1,C_2)$ from the lower-dimensional images becomes possible.
Controlling the projective images is manageable when the set of extreme points of the original sets is small---as in the case of $\ell_1^n$ (corresponding to the cross-polytope in $\R^n$) or the simplex.
In this way, we show for any Banach--Mazur position between the cones that the projection $P$ can be chosen to give images of the cones that are affinely equivalent to their bases.
A similar argument about controlling projections appears in the proof of the equality $\dbm(\ell_1^3, \ell_{\infty}^3) = \frac{9}{5}$ in \cite{kobosvarivoda}.
Besides this instance, the method seems to be novel for the study of (exact) Banach--Mazur distances.

Throughout this and the next section, we shall repeatedly use the following obvious observations:
For a $d$-dimensional convex body $B \subseteq \R^n$, where $d < n$, all $n$-dimensional cones of the form $C = \conv( B' \cup \{ v^1, \ldots v^{n-d} \} )$ with $B'$ an affine image of $B$ and appropriate vectors $v^1, \ldots, v^{n-d}$ are affinely equivalent.
The analogous statement about double cones $\conv( B' \cup \{ \pm v^1, \ldots, \pm v^{n-d} \} )$ for $0$-symmetric $B'$ holds as well.
In particular, we often assume that $B'$ is contained in the \cemph{linear span} $\lin \{e^1, \ldots, e^d\}$
and $v^1, \ldots, v^{n-d}$ are the remaining vectors from the canonical unit basis.

Next, applying a projection with $1$-dimensional kernel to a $d$-dimensional convex body reduces its dimension by at most $1$.
It does so if and only if the kernel of the projection is parallel to the affine span of the convex body.

We also use the simple lemma below.
It is proved here to avoid repeating the argument.
By $[x,y]$ we denote the closed \cemph{segment} connecting $x,y \in \R^n$.

\begin{lemma}
\label{lem:vertex_absorbing}
Let $B_1, B_2 \subseteq \R^n$ be convex,
let $v \in \R^n$ be a vector,
and let $T : \R^n \to \R^n$ be a linear operator.
If $B_1 \subseteq \conv( B_2 \cup \{v\} )$,
$v \notin B_1$,
and $T(v) \in T(B_1)$,
then
\[
    T(\conv(B_2 \cup \{v\}))
    = T(B_2).
\]
If $B_1$ and $B_2$ are $0$-symmetric, then the assumption $B_1 \subseteq \conv(B_2 \cup \{v\})$ can be weakened to $B_1 \subseteq \conv(B_2 \cup \{\pm v\})$.
\end{lemma}
\begin{proof}
For any set $A \subseteq \R^n$, it is clear that
\[
    T(\conv(B_2 \cup A))
    = \conv( T(B_2) \cup T(A) ),
\]
so it suffices to prove $T(v) \in T(B_2)$.
Now, choose any $u \in B_1$ with $T(v) = T(u)$.
If $B_1 \subseteq \conv(B_2 \cup \{v\})$,
there exist some $w \in B_2$ and $\lambda \in [0,1)$ such that
\begin{equation}
\label{eq:u_conv}
    u
    = (1-\lambda) w + \lambda v.
\end{equation}
Then
\[
    T(w)
    = T \left( \frac{u - \lambda v}{1-\lambda} \right)
    = \frac{T(u) - \lambda T(v)}{1-\lambda}
    = \frac{T(v) - \lambda T(v)}{1-\lambda}
    = T(v)
\]
verifies $T(v) \in T(B_2)$.

If $B_1$ and $B_2$ are $0$-symmetric and $B_1 \subseteq \conv(B_2 \cup \{\pm v\})$, it is straightforward to see that $u$ can still be given as in \eqref{eq:u_conv}, but possibly using $-v$ instead of $v$.
If it uses $v$, then $T(v) \in T(B_2)$ follows exactly like before.
Otherwise,
\[
    T \left( \frac{1-\lambda}{1+\lambda} w \right)
    = \frac{1-\lambda}{1+\lambda} T \left( \frac{u + \lambda v}{1-\lambda} \right)
    = \frac{T(u) + \lambda T(v)}{1+\lambda}
    = \frac{T(v) + \lambda T(v)}{1+\lambda}
    = T(v)
\]
and $\frac{1-\lambda}{1+\lambda} w \in [0,w] \subseteq B_2$ complete the proof.
\end{proof}

\begin{proof}[Proof of Theorem~\ref{thm:l1_sum}]
Since $(X \oplus_1 Y)^* = X^* \oplus_\infty Y^*$ holds for all finite-dimensional real normed spaces $X$ and $Y$,
it suffices to prove the result for the $\ell_1$-sum.
Moreover, by induction it is enough to show that for any $n$-dimensional real normed space $X = (\R^n, \| \cdot \|_X)$ with $d^* := \dbm(X, \ell_1^n) \leq 3$, we have
\[
    d
    := \dbm (X \oplus_1 \R, \ell_1^{n+1})
    = d^*.
\]

We identify the underlying vector space of $X \oplus_1 \R$ with $\R^{n+1}$ in the standard way, i.e.,
the vector space of $X$ is identified with the subspace $Y \subseteq \R^{n+1}$ given by $x_{n+1} = 0$.
To prove the upper bound $d \leq d^*$, we assume that $\|\cdot\|$ is a norm in $\R^n$ that is isometric to $\|\cdot\|_X$ and satisfies
\[
    \|x\|_1
    \leq \|x\|
    \leq d^* \|x\|_1
\]
for all $x \in \R^n$.
Then also
\[
    \|x\|_1 + |x_{n+1}|
    \leq \|x\| + |x_{n+1}|
    \leq d^* \|x\|_1 + |x_{n+1}|
    \leq d^*(\|x\|_1 + |x_{n+1}|)
\]
for all $x \in \R^n$ and all $x_{n+1} \in \R$.
It remains to establish the lower bound $d \geq d^*$.

Let $K = \conv \{ \pm v^1, \ldots, \pm v^{n+1} \} \subseteq \R^{n+1}$ be a linear image of the cross-polytope in $\R^{n+1}$ satisfying
\begin{equation}
\label{eq:incl_l1_sum}
    C
    \subseteq
    K \subseteq dC,
\end{equation}
where $C = \conv( B \cup \{ \pm e^{n+1} \} ) \subseteq \R^{n+1}$ is the unit ball of the space $X \oplus_1 \R$ for $B \subseteq Y$ the unit ball of $X$ embedded in $\R^{n+1}$.
Since $e^{n+1} \in C \subseteq K$, some vertex $v^i$ of $K$ satisfies $|v^i_{n+1}| \geq 1$.
We may suppose that $v^{n+1}_{n+1} \geq 1$.
For $x \in \R^{n+1}$, let $\pi_Y(x)=(x_1, \ldots, x_n, 0)$ denote its orthogonal projection onto $Y$.
Then
\[
    3
    \geq d^*
    \geq d
    \geq \|\pi_Y(v^{n+1})\|_X + |v^{n+1}_{n+1}|
    \geq \|\pi_Y(v^{n+1})\|_X + 1,
\]
so that $r := \|\pi_Y(v^{n+1})\|_X \leq 2$.
Next, we define a vector $w \in \R^{n+1}$ as
\[
    w
    := \begin{cases}
        \displaystyle
        e^{n+1} +\frac{1}{r \, v^{n+1}_{n+1}}\pi_Y(v^{n+1}),
            & \text{if } r > 0, \\
        e^{n+1},
            & \text{if } r = 0.
    \end{cases}
\]
Now, let us take a linear projection $P : \R^n \to Y$ given for $x \in \R^{n+1}$ as $P(x) = x - x_{n+1} w$.
Since $v^{n+1}_{n+1} \geq 1$,
we have $\|P(e^{n+1})\|_X = \|-\pi_Y(w)\|_X  \leq 1$
and, thus, $\pm P(e^{n+1}) \in B$.
It follows that
\[
    P(C)
    = \conv( P(B) \cup \{ \pm P(e^{n+1}) \} )
    = B.
\]
Applying $P$ to the inclusions \eqref{eq:incl_l1_sum} gives
\begin{equation}
\label{eq:double_cone_last_chain}
    B
    = P(C)
    \subseteq P(K)
    \subseteq d P(C)
    = d B.
\end{equation}
Moreover, if $r = 0$, then 
\[
    P(v^{n+1})
    = v^{n+1} - v^{n+1}_{n+1} e^{n+1}
    = \pi_Y(v^{n+1})
    = 0
    \in B.
\]
If instead $r \in (0, 2]$, then
\[
    P(v^{n+1})
    = v^{n+1} - v^{n+1}_{n+1} e^{n+1} - \frac{v^{n+1}_{n+1}}{r \, v^{n+1}_{n+1}} \pi_Y(v^{n+1})
    = \left( 1 - \frac{1}{r} \right) \pi_Y(v^{n+1}),
\]
so we have $\| P(v^{n+1}) \|_X = \left| 1 - \frac{1}{r} \right| \cdot r = |r - 1| \in [0,1]$.
In particular, we obtain $P(v^{n+1}) \in B$ in this case as well.
Since $P(B) = B \subseteq K$ and $v^{n+1} \notin B$, Lemma~\ref{lem:vertex_absorbing} shows that
\[
    P(K)
    = P(\conv\{\pm v^1, \ldots, \pm v^n \})
    = \conv \{ \pm P(v^1), \ldots, \pm P(v^n) \}
\]
is a centrally symmetric polytope with at most $2n$ vertices in the $n$-dimensional subspace $Y$.
As a projection of an $(n+1)$-dimensional convex body, it is $n$-dimensional and therefore a non-degenerate linear image of the $n$-dimensional cross-polytope.
Finally, 
\eqref{eq:double_cone_last_chain} and the definition of $d^*$ yield $d \geq d^*$,
finishing the proof.
\end{proof}

Unlike the symmetric case, the non-symmetric case has no a priori guarantee for any specific centre of homothety in an optimal Banach--Mazur position.
Thus, when considering the Banach--Mazur distance between single cones and the simplex, the position of the simplex is less constrained compared to the cross-polytope in the symmetric case.
The following lemma extends the projective argument to handle this additional degree of freedom.

\begin{lemma}
\label{lem:proj}
Let $B \subseteq \R^n$ be an $(n-1)$-dimensional convex body containing $0$
and contained in the subspace $Y = \{x \in \R^n: x_n = 0\}$.
Let $d \in [1, 2]$ and $u \in \R^n$ be such that the cone
\[
    C
    = \conv((B+u) \cup \{e^n+u\})
\]
contains $0$.
Then, for every vector $v \in dC$ with $v_n \geq u_n+1$,
there exists a linear projection $P : \R^n \to Y$ such that 
\[
    P(e^n) \in B
        \quad \text{and} \quad
    P(v-u) \in B.
\]
\end{lemma}
\begin{proof}
The assumption $0 \in C$ implies $u_n \in [-1, 0]$.
Let $\pi_Y : \R^n \to Y$ be the orthogonal projection onto the subspace $Y$.
Since $0 \in B \cap C$,
we have $0 \in \pi_Y(C) = B + \pi_Y(u)$ and therefore
\begin{equation}
\label{eq:u_orth_proj}
    -\pi_Y(u)
    \in B.
\end{equation}

The point $v \in dC$ lies in the section of the cone $dC$ by the affine hyperplane given by $x_n = v_n$.
This section is easily verified to equal
\[
    \lambda d B + (1-\lambda) d e^n + d u,
        \quad \text{where }
    \lambda
    := \frac{d-v_n+du_n}{d} \in [0, 1].
\]
Applying $\pi_Y$ yields
\[
    \pi_Y(v)
    \in \lambda d B + d \pi_Y(u)
    = (d-v_n+du_n) B + d \pi_Y(u).
\]
Hence, by $d + d u_n \geq v_n \geq u_n+1$, $u_n \in [-1, 0]$, $d \in [1, 2]$, $0 \in B$, and \eqref{eq:u_orth_proj},
we conclude that
\begin{equation}
\begin{split}
\label{eq:vu_orth_proj}
    \pi_Y(v-u)
    & \in (d-v_n+du_n)B + (d-1)\pi_Y(u)
    \subseteq (d-u_n-1+du_n) B + (d-1)\pi_Y(u)
    \\
    & = (d-1)((u_n+1)B + \pi_Y(u))
    \subseteq B + \pi_Y(u).
\end{split}
\end{equation}

Now, we define the linear projection $P : \R^n \to Y$ as $P(x) = x - x_n w$,
where
\[
    w
    := e^n + \frac{1}{v_n - u_n} \pi_Y(u).
\]
From $v_n - u_n \geq 1$, $0 \in B$, and \eqref{eq:u_orth_proj}, we obtain $P(e^n) = - \frac{1}{v_n - u_n} \pi_Y(u) \in [0,-\pi_Y(u)] \subseteq B$.
Moreover, from \eqref{eq:vu_orth_proj} it follows that
\[
    P(v-u)
    = v-u - (v_n-u_n) w
    = v-u - (v_n-u_n) e^n - \pi_Y(u)
    = \pi_Y(v-u) - \pi_Y(u)
    \in B.
\]
With both desired properties of $P$ verified, the proof is finished.
\end{proof}

\begin{proof}[Proof of Theorem~\ref{thm:simplex}]
By induction, it suffices to prove the case $m=1$.
Let $d^* = \dbm(B, \Delta^n)$ and $d = \dbm(C, \Delta^{n+1})$.
Further, let $Y \subseteq \R^{n+1}$ be the subspace given by $x_{n+1} = 0$.
To prove the upper bound $d \leq d^*$,
we may assume that $B_0 \subseteq Y$, $\Delta^n_0 \subseteq Y$ are affine copies of $B$ and $\Delta^n$, respectively,
with $0 \in B_0$ and $B_0 \subseteq \Delta^n_0 \subseteq d^* B_0$.
Then we have
\[
    C_0
    := \conv(B_0 \cup \{e^{n+1}\})
    \subseteq \conv(\Delta^n_0 \cup \{e^{n+1}\})
    =: \Delta^{n+1}_0.
\]
Moreover, $0 \in d^*B_0$ shows that also
\[
    \Delta^{n+1}_0
    = \conv(\Delta^n_0 \cup \{e^{n+1}\})
    \subseteq \conv(d^* B_0 \cup \{d^*  e^{n+1}\})
    = d^* C_0.
\]
Since the convex bodies $C_0$, $\Delta^{n+1}_0$ are affinely equivalent to $C$ and $\Delta^{n+1}$, respectively,
it follows that $d \leq d^*$.
It remains to prove the lower bound $d \geq d^*$.

For proving the lower bound, we may assume that $0 \in B \subseteq Y$,
$v^1 = e^{n+1}$,
and there exists a vector $u \in \R^{n+1}$ such that
$C = \conv((B+u) \cup\{e^{n+1}+u\})$ satisfies
\begin{equation}
\label{eq:incl_simplex}
	0
	\in C
	\subseteq K
	\subseteq dC,
\end{equation}
where $K = \conv \{w^1, \ldots, w^{n+2}\} \subseteq \R^{n+1}$ is a simplex.
The inequality $d \leq d^*$ from the previous part together with the assumption $d^* \leq 2$ yields $d \leq 2$.
Since $e^{n+1} + u \in C \subseteq K$,
some vertex $w^i$ of the simplex $K$, say $w^i = w^{n+2}$, satisfies $w^i_{n+1} \geq u_n + 1$.
By Lemma~\ref{lem:proj}, there exists a linear projection $P : \R^{n+1} \to Y$
with $P(e^{n+1}) \in B$ and $P(w^{n+2}) \in B + P(u)$.
In particular,
\[
    P(C)
    = \conv( (P(B) + P(u)) \cup \{P(e^{n+1}) + P(u)\} )
    = B + P(u).
\]
Applying $P$ to \eqref{eq:incl_simplex} yields
\begin{equation}
\label{eq:single_cone_last_chain}
    B + P(u)
    = P(C)
    \subseteq P(K)
    \subseteq d P(C)
    = d B + d P(u).
\end{equation}
By $B+u \subseteq K$, $w^{n+2} \notin B+u$, and $P(w^{n+2}) \in B + P(u) = P(B+u)$,
Lemma~\ref{lem:vertex_absorbing} shows that
\[
    P(K)
    = P(\conv\{ w^1, \ldots, w^{n+1} \})
    = \conv\{ P(w^1), \ldots, P(w^{n+1}) \}
\]
is a convex hull of $n+1$ points in the $n$-dimensional space $Y$.
It is a projection of an $(n+1)$-dimensional set,
so $P(K)$ is an $n$-dimensional simplex.
Hence, \eqref{eq:single_cone_last_chain} and $\dbm(B,\Delta^n) = d^*$ yield $d \geq d^*$, as required.
\end{proof}

\section{The Distance of (Double) Cones With Arbitrary Planar Symmetric Bases}
\label{sec:between_cones}

In this section, we further develop the projective method to study the Banach--Mazur distances between single and double cones with arbitrary planar symmetric bases.
Establishing the matching lower bounds on the distances again relies on finding linear projections in any Banach--Mazur position that simultaneously map both cones to sets affinely equivalent to their bases
or similarly easily handled.
When dealing with arbitrary cones, controlling the projective images is more involved compared to the previous section.
Our approach therefore requires a more detailed geometric analysis and, in some places, appropriate assumptions.

Our first goal is to verify Theorem~\ref{thm:3d_cones},
which states that the Banach--Mazur distance between two single cones in $\R^3$ with planar symmetric bases is equal to the distance between their bases.
The required $2$-dimensional analysis is carried out in the following lemma.

\begin{lemma}
\label{lem:triangles}
Let $y \in \R^2$ be such that $0$ lies in the triangle $T \subseteq \R^2$ with vertices $y - e^1$, $y + e^1$, and $y + e^2$.
Let $d \in \left[1, \frac{3}{2}\right]$ and let $S \subseteq \R^2$ be a triangle such that
\[
    T
    \subseteq S
    \subseteq d T,
\]
with the vertices $v, v', v'' \in \R^2$ of $S$ satisfying $v \in [d(y-e^1), dy]$ and $v_2' \geq v_2''$.
Then at least one of the following conditions holds:
\begin{enumerate}[(i)]
\item There exists $\mu \in \left[ 1, \frac{4}{3} \right)$ such that $v'_2 + \mu (v_2'' - v_2') = v_2$.
\item There exists $\mu \in \left[ \frac{1}{2}, \frac{4}{3} \right)$ such that the line through $y + e^2$ parallel to the line through $v'$ and $v''$ intersects the line through $y - e^1$ and $y + e^1$ at a point $z$ with $z_1 = y_1 + 2\mu - 1$.
\end{enumerate}
\end{lemma}

\noindent See Figure~\ref{fig:triangles} for an example of the situation in the lemma.
We note that the point $z$ does not need to lie in the triangle $d T$.

\begin{proof}
By $T \subseteq S$ and $v'_2 \geq v''_2 \geq v_2$,
we have $v'_2 \geq y_2+1$.
Moreover, $0 \in T$ yields $y_1 \in [-1,1]$ and $y_2 \in [-1,0]$.
Thus, $v' \in dT$ and $d \in [1, \frac{3}{2}]$ imply
\[
    v'_1
    \geq v'_1 - v'_2 + y_2+1
    \geq d (y_1 - y_2 - 1) + y_2 + 1
    = y_1 + (d-1) (y_1 - y_2 - 1)
    \geq y_1 - 1.
\]
Similarly, we can show that $v'_1 \leq y_1 + 1$.
Again using $T \subseteq S$ now gives $v''_1 \geq y_1+1$ and $v_1 \leq y_1 - 1$.
We also see for $y_2 = 0$ that $v_2 = v''_2 = 0$ necessarily holds, in which case condition (i) is satisfied.
Thus, we may assume $y_2 < 0$ and in particular $y_1 < 1$ from now on.

To estimate $v''_2$, we note that the line through $v$ and $v''$ passes under $y+e^1$.
With $v \in [d (y-e^1), dy]$ and $v_1 \leq y_1-1$,
this means that $v''$ does not lie above the line $L$ through $(y_1-1, dy_2)$ and $y+e^1$.
The line $L$ meets the line given by $x_1 = d(y_1+1)$ in
\[
    \left( d(y_1+1), y_2 + \frac{y_2 - d y_2}{2} (d (y_1+1) - (y_1+1)) \right)
    = \left( d(y_1+1), y_2 \left( 1 - \frac{(d-1)^2 (y_1+1)}{2} \right) \right).
\]
Since $v''$ does not lie above $L$,
$v''_1 \leq d (y_1+1)$ by $v'' \in dT$,
$d \in [1,\frac{3}{2}]$,
$y_1 \in [-1,1)$,
and $y_2 < 0$,
this means
\[
    v''_2
    \leq y_2 \left( 1 - \frac{(d-1)^2 (y_1+1)}{2} \right)
    \leq y_2 \left( 1 - \frac{y_1+1}{8} \right)
    < \frac{3}{4} y_2.
\]
If now $y_2 \geq - \frac{2}{5}$, then
\[
    v'_2 + \frac{4}{3} (v''_2 - v'_2)
    = \frac{4}{3} v''_2 - \frac{1}{3} v'_2
    < y_2 - \frac{1}{3} (y_2+1)
    = \frac{2}{3} y_2 - \frac{1}{3}
    \leq \frac{3}{2} y_2
    \leq d y_2
    = v_2.
\]
Hence, condition (i) is satisfied in this case, so we may assume $y_2 < - \frac{2}{5}$ from now on.

Next, let $w = (\frac{3}{2} (y_1+1) + \frac{3}{4} y_2, \frac{3}{4} y_2) \in [\frac{3}{2} (y+e^1), \frac{3}{2} (y+e^2)]$.
Since the line through $v'$ and $v''$ passes over $y+e^2$,
it has at most the same slope as the line through $y+e^2$ and $v''$.
Moreover, $y+e^2, v'' \in d T \subseteq \frac{3}{2} T$ and $v''_2 < \frac{3}{4} y_2$
show that the line through $y+e^2$ and $v''$ has a smaller slope than the line $L'$ through $y+e^2$ and $w$.
In particular, $z$ lies strictly to the left of the intersection point of $L'$ with the line through $y-e^1$ and $y+e^1$.
The latter intersection point can be computed to be
\[
    \left( y_1 + \frac{2 y_1 + 3 y_2 + 6}{y_2 + 4}, y_2 \right).
\]
Since $0 \in T$ implies $-y_1 + y_2 + 1 \geq 0$, and $y_2 \in [-1, - \frac{2}{5})$ by assumption, we obtain
\[
    z_1 - y_1
    < \frac{2 y_1 + 3 y_2 + 6}{y_2 + 4}
    \leq \frac{5 y_2 + 8}{y_2 + 4}
    < \frac{5}{3}.
\]
Thus, the unique $\mu \in \R$ with $z_1 = y_1 + 2 \mu - 1$ satisfies $\mu < \frac{4}{3}$.
Moreover, $v'_1 \leq y_1+1 \leq v''_1$ and $v'_2 \geq v''_2$ show that $z$ does not lie to the left of $y$, i.e., $\mu \geq \frac{1}{2}$.
This completes the proof.
\end{proof}

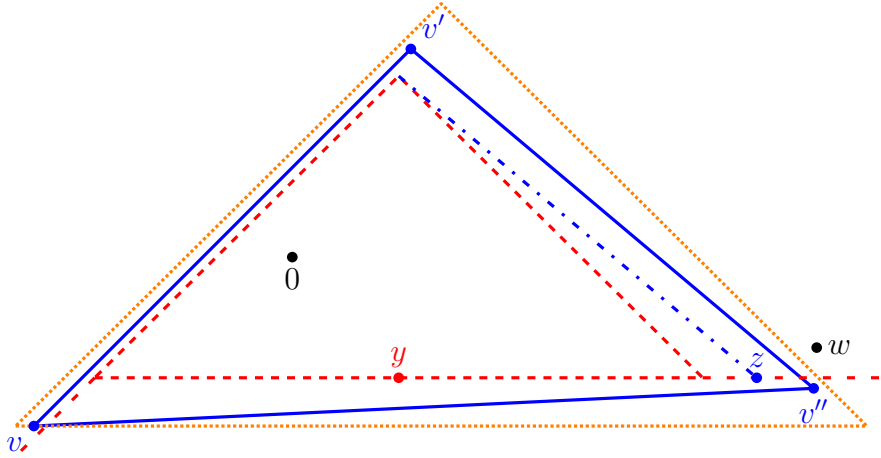
\begin{figure}[t]
\centering
\def\d{1.4}
\def\yx{0.35}
\def\yy{-0.4}
\def\vx{(\d*(\yx-1)+0.06)}
\def\vy{(\d*\yy)}
\def\vpx{(\d*\yx-0.1)}
\def\vpy{(\d*(\yy+1)-0.15)}
\def\vppx{(\d*(\yx+1)-0.175)}
\def\vppy{(\d*\yy+0.125)}
\def\zx{(\yx - (\vppx-\vpx)/(\vppy-\vpy))} 
\def\scale{4}
\begin{tikzpicture}[line join=round, scale=\scale]
\draw[very thick, blue] ({\vx},{\vy}) -- ({\vpx},{\vpy}) -- ({\vppx},{\vppy}) -- cycle;

\draw[very thick, red, dashed] (\yx-1,\yy) -- (\yx+1,\yy) -- (\yx,\yy+1) -- cycle;

\draw[very thick, orange, densely dotted]
  ({\d*(\yx-1)},{\d*\yy}) -- ({\d*(\yx+1)},{\d*\yy}) -- ({\d*\yx},{\d*(\yy+1)}) -- cycle;

\draw[very thick, blue, loosely dashdotted] ({\yx},{\yy+1}) -- ({\zx},{\yy});

\draw[very thick, red, loosely dashed] (\yx-1,\yy) -- (\yx-1.25,\yy+1-1.25);
\draw[very thick, red, loosely dashed] (\yx+1,\yy) -- (\yx+1.6,\yy);

\fill[red] (\yx,\yy) circle(2/\scale pt) node[anchor=south] {$y$};
\fill[blue] ({\zx},\yy) circle(2/\scale pt) node[anchor=south] {$z$};
\fill[blue] ({\vx},{\vy}) circle(2/\scale pt) node[anchor=north east] {$v$};
\fill[blue] ({\vpx},{\vpy}) circle(2/\scale pt) node[anchor=south west] {$v'$};
\fill[blue] ({\vppx},{\vppy}) circle(2/\scale pt) node[anchor=north] {$v''$};

\fill ({3/2*(\yx+1)+3/4*\yy},{3/4*\yy}) circle(2/\scale pt) node[anchor=west] {$w$};

\fill (0,0) circle(2/\scale pt) node[anchor=north] {$0$};
\end{tikzpicture}
\caption{
An example of the situation in Lemma~\ref{lem:triangles} and its proof for $d = \d$:
$T$ (red, dashed), $S$ (blue, solid), $d T$ (orange, dotted).
}
\label{fig:triangles}
\end{figure}

In the upcoming proof of Theorem~\ref{thm:3d_cones},
we also use the following well-known fact:
The Banach--Mazur distance between any centrally symmetric convex body in $\R^2$ and $\Delta^2$ is equal to $2$.
More generally, the Banach--Mazur distance between any centrally symmetric convex body in $\R^n$ and $\Delta^n$ equals $n$
(see, for example, \cite[Corollary~$5.8$]{glmp}).

\begin{proof}[Proof of Theorem~\ref{thm:3d_cones}]
Let $d^* := \dbm(B_1, B_2)$ and $d := \dbm(C_1, C_2)$.
Our first goal is to establish the upper bound $d \leq d^*$.
To this end, we may assume that $B_1, B_2$ are both $0$-symmetric and contained in the plane given by $x_3 = 0$,
the inclusions $B_1 \subseteq B_2 \subseteq d^* B_1$ hold,
and $v^1 = v^2 = e^3$.
Then
\[
    C_1
	= \conv(B_1 \cup \{e^3\})
	\subseteq \conv(B_2 \cup \{e^3\})
	= C_2
\]
and, since $0 \in d^* B_1$,
\[
	C_2
	= \conv(B_2 \cup \{e^3\})
	\subseteq \conv(d^*B_1 \cup \{d^*e^3\})
	=d^* C_1.
\]
It follows that $C_1 \subseteq C_2 \subseteq d^*C_1$,
which implies $d \leq d^*$.
It remains to prove the reverse inequality.
\pagebreak

Toward the lower bound, we may assume that $B_1$ is $0$-symmetric and contained in the plane given by $x_3 = 0$,
and there exist vectors $u, v \in \R^3$ such that $C_1 = \conv((B_1+u) \cup\{e^3+u\})$ satisfies
\begin{equation}
\label{eq:incl_3d_cones}
	0
	\in C_1
	\subseteq C_2
	= \conv(B_2 \cup \{v\})
	\subseteq d C_1.
\end{equation}
In this case, $u_3 \in [-1, 0]$.
By the first part of the proof and a result of Stromquist \cite{stromquist},
we have $d \leq d^* \leq \frac{3}{2}$.
To establish $d \geq d^*$,
we consider three cases according to the height $v_3$ of the vertex $v$ of $C_2$ relative to the base $B_1 + u$ of $C_1$.
From the inclusions above, it follows that $d u_3 + d \geq v_3 \geq d u_3$.

\textbf{Case 1.} $v_3 \geq u_3+1$. \\
By Lemma~\ref{lem:proj}, there exists a linear projection $P : \R^3 \to Y$, where $Y$ is the plane given by $x_3 = 0$,
such that $P(e^3) \in B_1$ and $P(v-u) \in B_1$.
Then 
\[
    P(C_1)
	= \conv( (P(B_1) + P(u)) \cup \{P(e^3) + P(u)\} )
    = B_1 + P(u),
\]
and applying $P$ to \eqref{eq:incl_3d_cones} gives
\begin{equation}
\label{eq:3d_cones_final_chain}
    B_1 + P(u)
    = P(C_1)
    \subseteq P(C_2)
    \subseteq d P(C_1)
    = d B_1 + d P(u).
\end{equation}
By $B_1+u \subseteq C_2$, $v \notin B_1+u$, and $P(v) \in B_1 + P(u) = P(B_1+u)$, Lemma~\ref{lem:vertex_absorbing} yields
\[
	P(C_2)
	= P(B_2).
\]
Consequently, $P(B_2)$ must be a non-degenerate linear image of $B_2$ in $Y$.
It follows from \eqref{eq:3d_cones_final_chain} that $d \geq \dbm(B_1, P(B_2)) = d^*$, as required.

\textbf{Case 2.} $u_3 + 1 > v_3 \geq u_3$. \\
Our goal is to show that this case cannot occur.
Since $e^3+u \in C_2 = \conv (B_2 \cup \{v\})$, the cone $C_2$ contains a point with $x_3$-coordinate at least $u_3+1$.
As $v_3 < u_3+1$, it must be some point from the base $B_2$.
Similarly, $B_1 + u \subseteq C_2$ implies that $C_2$ contains a point with $x_3$-coordinate at most $u_3$.
If no point from $B_2$ would satisfy this condition, then the section of $C_2$ for $x_3 = u_3$ would necessarily consist only of $v$.
Yet the corresponding section of $C_1$ is the $2$-dimensional set $B_1+u$,
which would contradict $C_1 \subseteq C_2$.
In consequence, $B_2$ must contain a point with $x_3$-coordinate at most $u_3$.
By convexity and $u_3 + 1 > v_3 \geq u_3$, there exists $y \in B_2$ with $y_3=v_3$ (see Figure~\ref{fig:3d_case2}).
Now, let $P : \R^3 \to Y$ be a linear projection onto the plane $Y := \lin(B_2-y)$ satisfying $P(v-y) = 0$.
Then $P(v) = P(y) \in P(B_2)$, so
\[
    P(C_2)
    = \conv (P(B_2) \cup \{P(v)\})
    = P(B_2)
    = P(B_2 - y) + P(y)
    = B_2 - y + P(y).
\]
Furthermore, $P(B_1)$ is a segment since the kernel $\lin\{v-y\}$ of $P$ is a subset of the plane given by $x_3 = 0$ containing $B_1$.
As $P(C_1) = \conv( (P(B_1)+P(u)) \cup \{P(e^3)+P(u)\} )$ is a projection of a $3$-dimensional set,
it must be $2$-dimensional and, thus, a triangle in $Y$.
Lastly, from applying $P$ to \eqref{eq:incl_3d_cones} to get
\[
    P(C_1)
    \subseteq P(C_2)
    = B_2 - y + P(y)
    \subseteq d P(C_1),
\]
we conclude that $\dbm(B_2,\Delta^2) \leq d \leq \frac{3}{2}$.
However, by the symmetry of $B_2$, this distance is actually equal to $2$, so we obtain the desired contradiction.
In summary, the case for $u_3 + 1 > v_3 \geq u_3$ never occurs.

\begin{figure}[ht]
\centering
\def\r{4.25}
\def\tx{0}
\def\ty{-1.25}
\def\tz{-1.67}
\def\ang{180}
\def\angg{330}
\def\ply{9}
\def\scale{0.75}
\tdplotsetmaincoords{75}{50}
\begin{tikzpicture}[line join=round, scale=\scale, tdplot_main_coords]
\coordinate (rA) at (\tx+\r,\ty+\r,\tz);
\coordinate (rB) at (\tx+\r,\ty-\r,\tz);
\coordinate (rC) at (\tx-\r,\ty-\r,\tz);
\coordinate (rD) at (\tx-\r,\ty+\r,\tz);
\coordinate (rE) at (\tx,\ty,\tz+2*\r);

\draw[very thick, orange, dashed] (rC) -- (rD) -- (rA);
\draw[very thick, orange, dashed] (rD) -- (rE);

\coordinate (F) at (0,-4,2/3);

\draw[very thick, blue, dashed]
  plot[domain=\ang:\angg,samples=60]
  ({0+1.815*sin(\x)},{1},{1/3+2*cos(\x)});

\coordinate (A) at (1,1,0);
\coordinate (B) at (1,-1,0);
\coordinate (C) at (-1,-1,0);
\coordinate (D) at (-1,1,0);

\coordinate (E) at (0,0,2);

\draw[very thick, red, dashed] (C) -- (D) -- (A);
\draw[very thick, red, dashed] (D) -- (E);

\draw[very thick, red] (B) -- (C) -- (E);

\draw[very thick, red] (1,\ply,0) -- (-1,\ply,0) -- (0,\ply,2) -- cycle;

\draw[very thick, blue]
  plot[domain=0:360,samples=360]
  ({0+1.815*sin(\x)},\ply,{1/3+2*cos(\x)});

\draw[very thick, orange]
  (\tx+\r,\ply,\tz) -- (\tx-\r,\ply,\tz) -- (\tx,\ply,\ty+2*\r) -- cycle;

\draw[thin, densely dashed] (F) -- (0,\ply,2/3);

\draw[very thick, red] (A) -- (B) -- (E) -- cycle;

\draw[very thick, blue]
  plot[domain=\angg-360:\ang,samples=60]
  ({0+1.815*sin(\x)},{1},{1/3+2*cos(\x)}) -- (F) -- cycle;

\draw[very thick, orange] (rB) -- (rE) -- (rA) -- (rB) -- (rC) -- (rE);
\end{tikzpicture}
\caption{
An example of the situation in Case $2$ in the proof of Theorem~\ref{thm:3d_cones}:
$C_1$ (red), $C_2$ (blue), $d C_2$ (orange).
The vertex of the middle cone is projected onto the projection image of the cone's base.
}
\label{fig:3d_case2}
\end{figure}
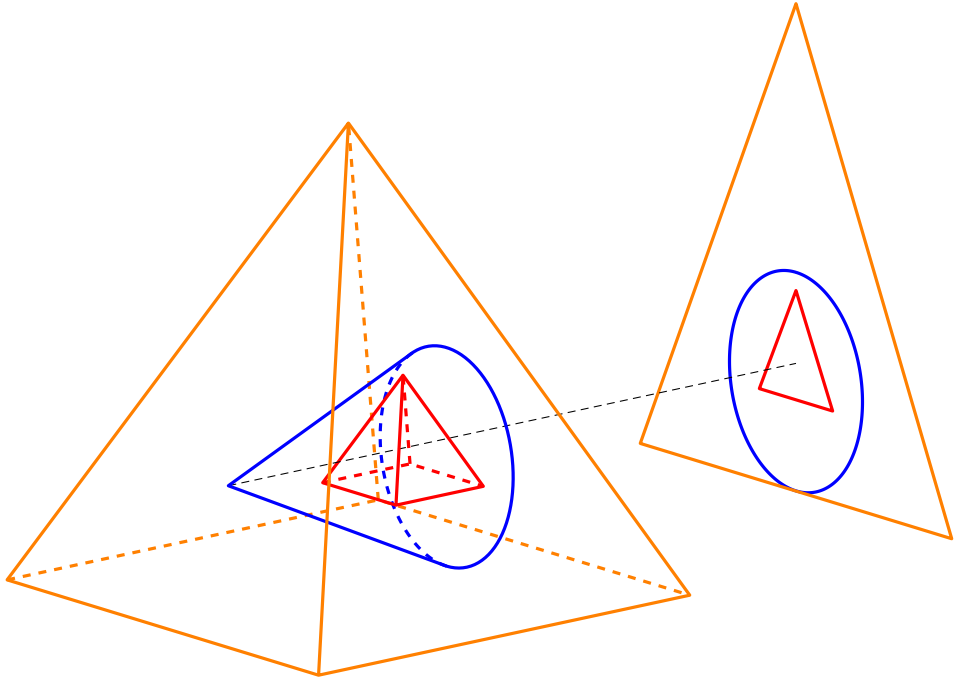

\textbf{Case 3}. $v_3 < u_3$. \\
Similar to the second case, our goal here is to use linear projections onto appropriately chosen planes to reach a contradiction with the Banach--Mazur distance to the triangle.
However, the situation now requires a more detailed analysis
since obtaining only triangles or non-degenerate affine transformations of the bases as the projections of both $C_1$ and $C_2$ simultaneously is not necessarily possible anymore.
Instead, we use Lemma~\ref{lem:triangles} to control the shapes of the possible projection images.

In the following, we may assume that all points of the base $B_2$ have strictly larger $x_3$-coordinate than $v$; otherwise, the reasoning from the previous case would still apply.
Furthermore, shifting $C_1$ and $C_2$ upward by $\frac{v_3 - d u_3}{d-1} e^3$ (and consequently shifting $d C_1$ by $d \frac{v_3 - d u_3}{d-1} e^3$),
where
\[
    v_3 + \frac{v_3 - d u_3}{d-1}
    = d u_3 + d \frac{v_3 - d u_3}{d-1},
\]
preserves the assumption $v_3 < u_3$ and the inclusions $0 \in C_1 \subseteq C_2 \subseteq d C_1$ by $0 \in B_1$ and the previous assumption.
Thus, we may also suppose that the vertex $v$ of $C_2$ lies in the base $d (B_1+u)$ of $d C_1$.

Next, we claim that there exists a plane $V \subseteq \R^3$ such that $0, e^3+u, v \in V$ and the intersection $(B_1 + u) \cap V$ is a segment.
Indeed, if the line through $e^3+u$ and $v$ meets $0$, then we can take any plane containing this line and intersecting the relative interior of $B_1 + u$.
Otherwise, the ray from $e^3+u \neq 0$ through $0$ meets $B_1 + u$ by $0 \in C_1$, in a point different from $\frac{1}{d} v$.
Since $\frac{1}{d} v \in B_1+u$ by our second assumption above, this means the plane
$\lin\{e^3+u, v\}$ intersects $B_1 + u$ in a segment.
In either case, a plane $V$ with all required properties exists.

Now, $T := C_1 \cap V$ is a triangle.
Furthermore, the intersection $B_2 \cap V$ is at most a segment, as $v \in V \setminus B_2$.
The containments
\[
    T
    = C_1 \cap V
    \subseteq C_2 \cap V
    \subseteq (dC_1) \cap V
    = dT
\]
following from \eqref{eq:incl_3d_cones} thus imply that the intersection $S := C_2 \cap V$ is $2$-dimensional,
that is, $B_2 \cap V$ is a segment and $S$ is also a triangle.
We can now choose a coordinate system within $V$ that makes Lemma~\ref{lem:triangles} applicable.
Let us write $(B_1+u) \cap V = [y-f^1, y+f^1]$ for some vectors $y, f^1 \in V$, such that the point $y$ corresponds to the point $y$ in the lemma, the vector $f^1$ corresponds to the canonical unit basis vector $e^1 \in \R^2$,
and the vector $f^2 \in V$ with $y+f^2 = e^3+u$ corresponds to the vector $e^2 \in \R^2$.
Similarly, the vertex $v$ of $C_2$ corresponds to the point $v$ in the lemma and we can choose $v', v'' \in V$ such that $B_2 \cap V = [v', v'']$ and all conditions of the lemma are satisfied for the triangles $S$ and $T$.
In the following, we consider separate cases given by the two conditions in the lemma.

Let us suppose that condition (i) of Lemma~\ref{lem:triangles} holds
and let $P : \R^3 \to Y$ be the orthogonal projection onto the plane $Y$ orthogonal to $f^1$, i.e., $P(f^1) = 0$ (see Figure~\ref{fig:3d_case3i}).
Note that $\lin\{f^1\}$ is the first coordinate axis of the chosen coordinate system in $V$,
which is the common line of the planes $V$ and $\lin(B_1)$.
Therefore, both of these planes are projected onto lines in $Y$.
Similar to the second case,
the set $P(C_1)$ is thus a triangle in the plane $Y$.
While the projection $P(C_2)$ might be affinely non-equivalent to $B_2$ this time,
condition (i) of Lemma~\ref{lem:triangles} nonetheless allows us to upper bound $\dbm(P(C_2), B_2)$.
Since $P(v), P(v'), P(v'')$ lie on the line $P(V)$ in $Y$,
it yields some $\mu \in [1,\frac{4}{3})$ such that
\[
    P(v)
    = \mu(P(v'')-P(v')) + P(v').
\]
Consequently, by $v'' \in B_2$,
\[
    P(v)
    \in \mu (P(B_2) - P(v')) + P(v').
\]
Furthermore, $\mu \geq 1$ and $v' \in B_2$ yield $B_2 \subseteq \mu(B_2-v') + v'$, so that
\[
    P(B_2)
    \subseteq \mu(P(B_2)-P(v')) + P(v').
\]
Altogether, we obtain that
\begin{equation}
\label{eq:incl_mu}
    P(B_2)
    \subseteq P(C_2)
    = \conv (P(B_2) \cup \{P(v)\})
    \subseteq \mu(P(B_2) - P(v')) + P(v').
\end{equation}
In particular, since the set $P(C_2)$ is $2$-dimensional as a projection of a $3$-dimensional convex body,
the set $P(B_2)$ must also be $2$-dimensional.
Therefore, $P(B_2)$ is a non-degenerate linear image of $B_2$,
and the inclusions \eqref{eq:incl_mu} together with $P(C_1) \subseteq P(C_2) \subseteq d P(C_1)$ from \eqref{eq:incl_3d_cones} show that
\[
    \dbm(B_2, \Delta^2)
    \leq \dbm(P(B_2), P(C_2)) \cdot \dbm(P(C_2), P(C_1))
    \leq \mu \cdot d
    < \frac{4}{3} \cdot \frac{3}{2}
    = 2.
\]
This contradicts the Banach--Mazur distance between any planar symmetric convex body and $\Delta^2$ being equal to $2$,
as desired.
\pagebreak

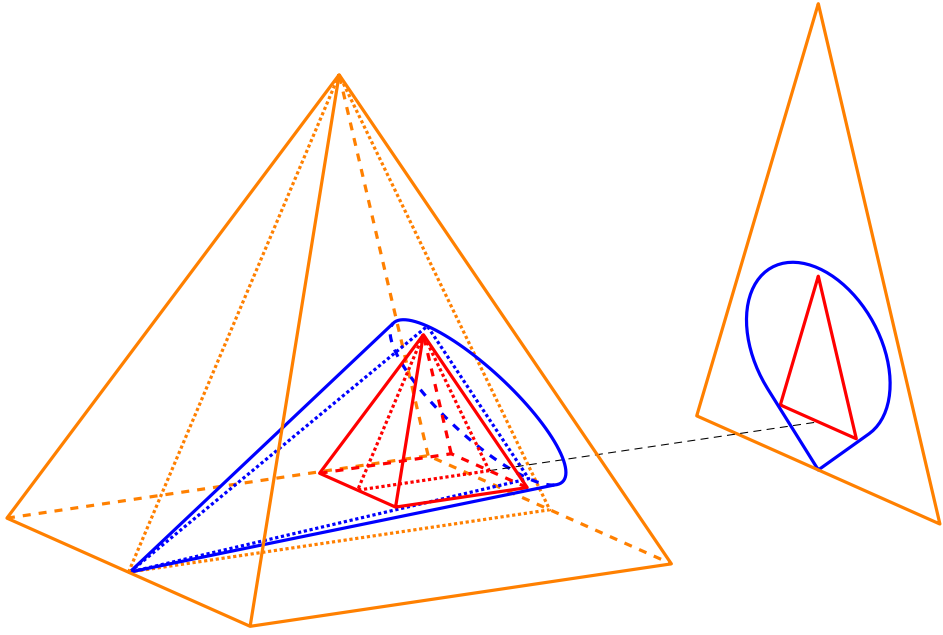
\begin{figure}[ht]
\centering
\def\r{3.2}
\def\tx{0}
\def\ty{-1.28}
\def\tz{-0.66}
\def\ang{245}
\def\angg{50}
\def\ply{6}
\def\scale{1}
\tdplotsetmaincoords{75}{60}
\begin{tikzpicture}[line join=round, scale=\scale, tdplot_main_coords]
\coordinate (rA) at (\tx+\r,\ty+\r,\tz);
\coordinate (rB) at (\tx+\r,\ty-\r,\tz);
\coordinate (rC) at (\tx-\r,\ty-\r,\tz);
\coordinate (rD) at (\tx-\r,\ty+\r,\tz);
\coordinate (rE) at (\tx,\ty,\tz+2*\r);

\draw[very thick, orange, dashed] (rC) -- (rD) -- (rA);
\draw[very thick, orange, dashed] (rD) -- (rE);

\draw[very thick, orange, densely dotted]
  (0,\ty-\r,\tz) -- (0,\ty+\r,\tz) -- (0,\ty,\tz+2*\r);

\coordinate (F) at (0,-4.46,-0.66);

\draw[very thick, blue, dashed]
  plot[domain=\ang-360:\angg,samples=60]
  ({-1.89*cos(\x)},{0.83-0.77*sin(\x)},{0.96+1.15*sin(\x)});

\draw[very thick, blue, densely dotted] (F) -- (0,0.06,2.11) -- (0,1.6,-0.19) -- cycle;

\draw[very thick, red] (1,\ply,0) -- (-1,\ply,0) -- (0,\ply,2) -- cycle;

\draw[very thick, blue]
  plot[domain=-45:225,samples=60]
  ({-1.89*cos(\x)},\ply,{0.96+1.15*sin(\x)}) -- (\tx,\ply,\tz) -- cycle;

\draw[very thick, orange]
  (\tx+\r,\ply,\tz) -- (\tx-\r,\ply,\tz) -- (\tx,\ply,\tz+2*\r) -- cycle;

\draw[thin, densely dashed] (0,1,0) -- (0,\ply,0);

\coordinate (A) at (1,1,0);
\coordinate (B) at (1,-1,0);
\coordinate (C) at (-1,-1,0);
\coordinate (D) at (-1,1,0);

\coordinate (E) at (0,0,2);

\draw[very thick, red, dashed] (C) -- (D) -- (A);
\draw[very thick, red, dashed] (D) -- (E);

\draw[very thick, red, densely dotted] (0,1,0) -- (0,-1,0) -- (0,0,2) -- cycle;

\draw[very thick, red] (B) -- (E) -- (A) -- (B) -- (C) -- (E);

\draw[very thick, blue]
  plot[domain=\angg:\ang,samples=60]
  ({-1.89*cos(\x)},{0.83-0.77*sin(\x)},{0.96+1.15*sin(\x)}) -- (F) -- cycle;

\draw[very thick, orange, densely dotted] (0,\ty-\r,\tz) -- (0,\ty,\tz+2*\r);

\draw[very thick, orange] (rB) -- (rE) -- (rA) -- (rB) -- (rC) -- (rE);
\end{tikzpicture}
\caption{
An example of the situation for condition (i) of Lemma~\ref{lem:triangles} in Case $3$ in the proof of Theorem~\ref{thm:3d_cones}:
$C_1$ (red), $C_2$ (blue), $d C_2$ (orange).
The triangles $T = C_1 \cap V$, $S = C_2 \cap V$, and $d T$
are marked as dotted sections of the respective cones.
The projection along the first coordinate axis in the chosen coordinate system in $V$
maps the middle cone to a set almost affinely equivalent to its base.
}
\label{fig:3d_case3i}
\end{figure}

Finally, let us assume that condition (ii) of Lemma~\ref{lem:triangles} is true.
Let $w = v'' - v'$ and let $Y$ be the plane given by $x_3 = 0$.
Note that the ray from $y+f^2 = e^3+u \notin Y + u$ with direction $w$ meets the line through $y-f^1$ and $y+f^1$ in a point in $Y+u$, so $w_3 \neq 0$.
Thus, we can take a linear projection $P : \R^3 \to Y$ satisfying $P(w) = 0$ (see Figure~\ref{fig:3d_case3ii}).
Since $w$ is parallel to $\lin(B_2)$, the set $P(B_2)$ is a segment.
Moreover, we know that $P(C_2)$ must be a $2$-dimensional set as a projection of a $3$-dimensional convex body, so $P(C_2)$ is a triangle in $Y$.
Similar to before, the set $P(C_1)$ might be affinely non-equivalent to $B_1$, but its Banach--Mazur distance to $B_1$ can be bounded.
Condition (ii) of Lemma~\ref{lem:triangles} implies that
there exists $\mu \in [\frac{1}{2}, \frac{4}{3})$ such that
\[
    P(e^3 + u)
    = P(y + f^2)
    = P(y + (2 \mu - 1) f^1).
\]
Writing $\mu^* = \max \{ \mu, 1 \} \in [1,\frac{4}{3})$ and using $y - f^1, y + f^1 \in B_1 + u$ gives
\begin{align*}
    P(e^3 + u)
    & = \mu(P(y+f^1) - P(y-f^1)) + P(y-f^1)
    \\
    & \in \mu (P(B_1+u) - P(y-f^1)) + P(y-f^1)
    \\
    & \subseteq \mu^* (P(B_1 + u) - P(y-f^1)) + P(y-f^1).
\end{align*}
Furthermore, we also have
\[
    P(B_1+u)
    \subseteq \mu^* (P(B_1+u) - P(y-f^1)) + P(y-f^1),
\]
and consequently
\[
    P(B_1+u)
    \subseteq P(C_1)
    = \conv(P(B_1+u) \cup \{P(e^3+u)\})
    \subseteq \mu^* (P(B_1+u) - P(y-f^1)) + P(y-f^1).
\]
Since $P(C_1)$ is $2$-dimensional as projection of a $3$-dimensional convex body,
we see that also $P(B_1+u)$ is $2$-dimensional and therefore a non-degenerate affine image of $B_1$.
Altogether, with $P(C_1) \subseteq P(C_2) \subseteq d P(C_1)$ from \eqref{eq:incl_3d_cones} we obtain that
\[
    \dbm(B_1, \Delta^2)
    \leq \dbm(P(B_1+u), P(C_1)) \cdot \dbm(P(C_1), P(C_2))
    \leq \mu^* \cdot d
    < \frac{4}{3} \cdot \frac{3}{2} = 2,
\]
which again contradicts that $B_1$ is symmetric and its Banach--Mazur distance to $\Delta^2$ is $2$.
We obtained contradictions in all situations except Case $1$,
so $d \geq d^*$ follows as claimed.
\end{proof}

Let us point out that symmetry of the bases is not strictly necessary for the validity of Theorem~\ref{thm:3d_cones}.
For the above proof to work, it is enough if $B_1, B_2 \subseteq \R^3$ are $2$-dimensional convex bodies that satisfy
\[
    \dbm(B_1,B_2) \leq \min \left\{ \frac{3}{2}, \, \frac{3}{4} \dbm(B_1,\Delta^2), \, \frac{3}{4} \dbm(B_2,\Delta^2) \right\}.
\]

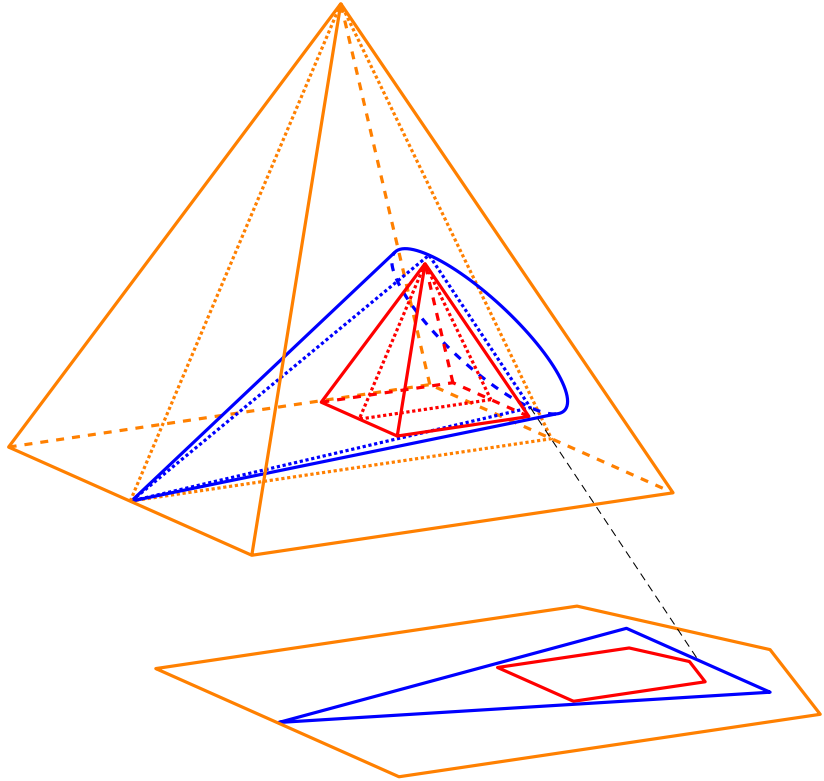
\begin{figure}[t]
\centering
\def\r{3.2}
\def\tx{0}
\def\ty{-1.28}
\def\tz{-0.66}
\def\ang{245}
\def\angg{50}
\def\plz{-4}
\def\scale{1}
\tdplotsetmaincoords{75}{60}
\begin{tikzpicture}[line join=round, scale=\scale, tdplot_main_coords]
\coordinate (rA) at (\tx+\r,\ty+\r,\tz);
\coordinate (rB) at (\tx+\r,\ty-\r,\tz);
\coordinate (rC) at (\tx-\r,\ty-\r,\tz);
\coordinate (rD) at (\tx-\r,\ty+\r,\tz);
\coordinate (rE) at (\tx,\ty,\tz+2*\r);

\draw[very thick, orange, dashed] (rC) -- (rD) -- (rA);
\draw[very thick, orange, dashed] (rD) -- (rE);

\draw[very thick, orange, densely dotted]
  (0,\ty-\r,\tz) -- (0,\ty+\r,\tz) -- (0,\ty,\tz+2*\r);

\coordinate (F) at (0,-4.46,-0.66);

\draw[very thick, blue, dashed]
  plot[domain=\ang-360:\angg,samples=60]
  ({-1.89*cos(\x)},{0.83-0.77*sin(\x)},{0.96+1.15*sin(\x)});

\draw[very thick, blue, densely dotted] (F) -- (0,0.06,2.11) -- (0,1.6,-0.19) -- cycle;

\draw[very thick, red]
  (1,{1-\plz*(1.54/2.3)},\plz) -- (0,{(2-\plz)*(1.54/2.3)},\plz)
  -- (-1,{1-\plz*(1.54/2.3)},\plz) -- (-1,{-1-\plz*(1.54/2.3)},\plz)
  -- (1,{-1-\plz*(1.54/2.3)},\plz) -- cycle;

\draw[very thick, blue]
  (0,{-4.46-(0.66+\plz)*(1.54/2.3)},\plz) -- (-1.89,{0.83-(-0.96+\plz)*(1.54/2.3)},\plz)
  -- (1.89,{0.83-(-0.96+\plz)*(1.54/2.3)},\plz) -- cycle;

\draw[very thick, orange]
  (\tx+\r,{\ty+\r-(\plz-\tz)*(1.54/2.3)},\plz) -- (\tx,{\ty-(\plz-\tz-2*\r)*(1.54/2.3)},\plz)
  -- (\tx-\r,{\ty+\r-(\plz-\tz)*(1.54/2.3)},\plz)
  -- (\tx-\r,{\ty-\r-(\plz-\tz)*(1.54/2.3)},\plz)
  -- (\tx+\r,{\ty-\r-(\plz-\tz)*(1.54/2.3)},\plz) -- cycle;

\draw[thin, densely dashed] (0,1.6,-0.19) -- (0,{0.83-(-0.96+\plz)*(1.54/2.3)},\plz);

\coordinate (A) at (1,1,0);
\coordinate (B) at (1,-1,0);
\coordinate (C) at (-1,-1,0);
\coordinate (D) at (-1,1,0);

\coordinate (E) at (0,0,2);

\draw[very thick, red, dashed] (C) -- (D) -- (A);
\draw[very thick, red, dashed] (D) -- (E);

\draw[very thick, red, densely dotted] (0,1,0) -- (0,-1,0) -- (0,0,2) -- cycle;

\draw[very thick, red] (B) -- (E) -- (A) -- (B) -- (C) -- (E);

\draw[very thick, blue]
  plot[domain=\angg:\ang,samples=60]
  ({-1.89*cos(\x)},{0.83-0.77*sin(\x)},{0.96+1.15*sin(\x)}) -- (F) -- cycle;

\draw[very thick, orange, densely dotted] (0,\ty-\r,\tz) -- (0,\ty,\tz+2*\r);

\draw[very thick, orange] (rB) -- (rE) -- (rA) -- (rB) -- (rC) -- (rE);
\end{tikzpicture}
\caption{
An example of the situation for condition (ii) of Lemma~\ref{lem:triangles} in Case $3$ in the proof of Theorem~\ref{thm:3d_cones}:
$C_1$ (red), $C_2$ (black), $d C_2$ (orange).
The projection with kernel parallel to $w$
maps the inner and outer cones to sets almost affinely equivalent to their bases.
}
\label{fig:3d_case3ii}
\end{figure}

Next, we turn to Theorem~\ref{thm:sym_cones} on double cones with planar symmetric bases.
The reasoning is quite similar to the one used in the proof of Theorem~\ref{thm:l1_sum}.
\pagebreak

\begin{proof}[Proof of Theorem~\ref{thm:sym_cones}]
To prove the upper bound, suppose that the norms $\|\cdot\|_{X_1}$ and $\|\cdot\|_{X_2}$ in $\R^2$ are chosen so that
\[
    \|x'\|_{X_1}
    \leq \|x'\|_{X_2}
    \leq d^* \|x'\|_{X_1}
\]
for all $x' \in \R^2$, where $d^* := \dbm(X_1,X_2)$.
Then, for any $x = (x', x'') \in \R^{m+2}$ with $x' \in \R^2$ and $x'' \in \R^m$, we have
\[
    \|x\|_{X_1 \oplus_1 \ell_1^m}
    = \|x'\|_{X_1} + \|x''\|_{1}
    \leq \|x'\|_{X_2} + \|x''\|_{1}
    = \|x\|_{X_2 \oplus_1 \ell_1^m}
\]
and
\[
    \|x\|_{X_2 \oplus_1 \ell_1^m}
    = \|x'\|_{X_2} + \|x''\|_{1}
    \leq d^*\|x'\|_{X_1} + d^*\|x''\|_{1}
    = d^*\|x\|_{X_1 \oplus_1 \ell_1^m}.
\]
Thus,
\[
    \dbm(X_1 \oplus_1 \ell_1^{m}, X_2 \oplus_1 \ell_1^{m})
    \leq d^*
\]
for every integer $m \geq 1$.

To prove the lower bound,
we proceed by induction on $m \geq 0$,
with the convention that $X \oplus_1 \ell_1^0 = X$.
Therefore, there is nothing to show for $m = 0$.
Now, assume $m \geq 1$ and that the lower bound holds for $m-1$, that is,
\begin{equation}
\label{eq:sym_cones_ind}
    \dbm(X_1 \oplus_1 \ell_1^{m-1}, X_2 \oplus_1 \ell_1^{m-1})
    \geq d^*.
\end{equation}
Without loss of generality, we may assume that 
\begin{equation}
\label{eq:dist_assumption}
	\dbm(X_1, \ell_1^2)
    \geq d^*.
\end{equation}
Let $C_1, C_2$ be unit balls of spaces isometric to $X_1 \oplus_1 \ell_1^m$ and $X_2 \oplus_1 \ell_1^m$, respectively, such that
\begin{equation}
\label{eq:incl_sym_cones}
    C_1
    \subseteq C_2
    \subseteq d C_1
\end{equation}
for $d := \dbm(X_1 \oplus_1 \ell_1^{m}, X_2 \oplus_1 \ell_1^{m})$.
We may suppose that $C_1 = \conv( B_1 \cup \{\pm e^3, \ldots, \pm e^{m+2} \} )$, where $B_1 \subseteq \lin \{e^1, e^2\}$ is the unit ball of $X_1$,
and $C_2=\conv( B_2 \cup \{\pm v^3, \ldots, \pm v^{m+2} \} )$ for some $2$-dimensional $0$-symmetric convex body $B_2$ affinely equivalent to the unit ball of $X_2$ and vectors $v^3, \ldots, v^{m+2} \in \R^{m+2}$.

Since $e^{m+2} \in C_1 \subseteq C_2$, there exists a point $v \in C_2$ with $v_{m+2} \geq 1$.
We define a vector $w := \frac{v}{v_{m+2}}$.
Let $Y \subseteq \R^{m+2}$ be the subspace given by $x_{m+2} = 0$ and let $P : \R^{m+2} \to Y$ be the linear projection onto $Y$ with $P(x) = x - x_{m+2} w$ for $x \in \R^{m+2}$. Then $P(v) = 0$ and
\[
    P(e^{m+2})
    = e^{m+2} - w
    = -\pi_Y(w)
\]
for $\pi_Y$ the orthogonal projection onto $Y$.
Writing $C_1' = \conv(B_1 \cup \{\pm e^3, \ldots, \pm e^{m+1} \})$, we note that
\[
    2
    \geq d^*
    \geq d
    \geq \|\pi_Y(v)\|_{C_1'} + |v_{m+2}|.
\]
Consequently, by $v_{m+2} \geq 1$,
\[
    \|P(e^{m+2})\|_{C_1'}
    = \|-\pi_Y(w)\|_{C_1'}
    \leq \frac{2-v_{m+2}}{v_{m+2}}
    \leq 1.
\]
It follows that
\[
    P(C_1)
    = \conv(P(C_1') \cup \{ \pm P(e^{m+2}) \})
    = C_1'.
\]
Applying $P$ to \eqref{eq:incl_sym_cones} yields
\begin{equation}
\label{eq:sym_cones_chain}
    P(C_1)
    \subseteq P(C_2)
    \subseteq d P(C_1).
\end{equation}
\pagebreak

To handle $P(C_2)$, we consider two cases for the choice of $v$.
If $v$ is one of the $\pm v^i$ for $i = 3, \ldots, m+2$, say $v = v^{m+2}$, then $P(v) = 0$ implies
\[
    P(C_2)
    = \conv( P(B_2) \cup \{ \pm P(v^3), \ldots, \pm P(v^{m+1}) \})
    =: C_2'.
\]
Since $C_2'$ is $(m+1)$-dimensional as a projection of an $(m+2)$-dimensional convex body,
$P(B_2)$ is $2$-dimensional by
\[
    m+1
    = \dim(C_2')
    \leq \dim(P(B_2)) + m-1
    \leq m+1.
\]
Consequently,
$P(B_2)$ is affinely equivalent to $B_2$,
and $C_2'$ is the unit ball of a space isometric to $X_2 \oplus_1 \ell_1^{m-1}$.
With \eqref{eq:sym_cones_ind} and \eqref{eq:sym_cones_chain}, the claim follows by
\[
    d
    \geq \dbm(P(C_1), P(C_2))
    = \dbm(X_1 \oplus_1 \ell_1^{m-1}, X_2 \oplus_1 \ell_1^{m-1})
    \geq d^*.
\]

Now, assume that $v$ belongs to $B_2$.
Since $P(v) = 0$, the projection $P(B_2)$ is a segment.
Consequently,
\[
    P(C_2)
    = \conv( P(B_2) \cup \{ \pm P(v^3), \ldots, \pm P(v^{m+2}) \})
\]
is a centrally symmetric polytope with at most $2(m+1)$ vertices.
Since it is $(m+1)$-dimensional as a projection of an $(m+2)$-dimensional convex body, $P(C_2)$ is a non-degenerate linear image of the $(m+1)$-dimensional cross-polytope.
Therefore, Theorem~\ref{thm:l1_sum} implies
\[
    \dbm(P(C_1), P(C_2))
    = \dbm(X_1 \oplus_1 \ell_1^{m-1}, \ell_1^{m+1})
    = \dbm(X_1, \ell_1^2).
\]
With the assumption \eqref{eq:dist_assumption} and the inclusions \eqref{eq:sym_cones_chain}, we obtain
\[
    d
    \geq \dbm(P(C_1), P(C_2))
    = \dbm(X_1, \ell_1^2)
    \geq d^*.
\]
Since $v$ can always be chosen in one of the two ways considered above, the induction step and the theorem follow.
\end{proof}

Before we proceed with the corollaries of Theorem~\ref{thm:sym_cones},
let us discuss the necessity of the assumption on the distances to $\ell_1^2$ in the theorem.
If the assumption fails, meaning that $\dbm(X_1,\ell_1^2) < \dbm(X_1,X_2)$ and $\dbm(X_2,\ell_1^2) < \dbm(X_1,X_2)$,
then the conclusion of the theorem is wrong,
at least for $m \geq 2$.
Indeed, suppose that $m \geq 2$ and there exists some $1 \leq d < \dbm(X_1,X_2)$ such that  $\dbm(X_i, \ell_1^2) \leq d$ for $i=1, 2$. Without loss of generality we may assume that
\[
    \frac{1}{d} \|x\|_1 \leq \|x\|_{X_1} \leq \|x\|_1
        \quad \text{and} \quad
    \|x\|_1 \leq \|x\|_{X_2} \leq d \|x\|_1
\]
for all $x \in \R^2$.
Then, for all $(x,y,z) \in \R^{m+2}$ with $x,y \in \R^2$ and $z \in \R^{m-2}$,
\[
    \|(x,y,z)\|_{X_1 \oplus_1 \ell_1^m}
    = \|x\|_{X_1} + \|y\|_1 + \|z\|_1
    \leq \|x\|_1 + \|y\|_{X_2} + \|z\|_1
    = \|(x,y,z)\|_{\ell_1^2 \oplus_1 X_2 \oplus_1 \ell_1^{m-2}}.
\]
Moreover,
\[
    \|(x,y,z,)\|_{\ell_1^2 \oplus_1 X_2 \oplus_1 \ell_1^{m-2}}
    = \|x\|_1 + \|y\|_{X_2} + \|z\|_1
    \leq d \|x\|_{X_1} + d \|y\|_1 + d \|z\|_1
    = d \|(x,y,z)\|_{X_1 \oplus_1 \ell_1^m}.
\]
Since $\ell_1^2 \oplus_1 X_2 \oplus_1 \ell_1^{m-2}$ and $X_2 \oplus_1 \ell_1^m$ are clearly isometric,
we obtain
\[
    \dbm(X_1 \oplus_1 \ell_1^m, X_2 \oplus_1 \ell_1^m)
    \leq d
    < \dbm(X_1,X_2),
\]
which shows that the conclusion of Theorem~\ref{thm:sym_cones} does not hold in this case.
This naturally leads to the question whether the assumption can ever fail.

\begin{quest}
\label{que:l1}
Does the inequality
\begin{equation}
\label{eq:quest_ineq}
    \dbm(X_1,X_2)
    \leq \max \left\{ \dbm(X_1, \ell_1^2), \dbm(X_2, \ell_1^2) \right\}
\end{equation}
hold for all $2$-dimensional real normed spaces $X_1$ and $X_2$?
\end{quest}

Clearly, a positive answer to this question would allow us to drop the additional assumption in Theorem~\ref{thm:sym_cones} and would lead to an isometric embedding of $\BM^2_s$ into $\BM^n_s$ for all $n \geq 3$. We discuss the plausibility and some further consequences of a positive answer in the following remark.

\begin{remark}
\label{rem:l1}
While the inequality \eqref{eq:quest_ineq} might seem implausible at first, we are not aware of any counterexamples.
Instead, we can provide some circumstantial evidence that might speak to its validity.

First, there are several instances in the real planar setting in which $\ell_1^2$ turns out to be a maximizer of the Banach--Mazur distance.
These include the diameter of the compactum, i.e., the maximal distance between arbitrary spaces (determined by Stromquist \cite{stromquist}),
the maximal distance to the Euclidean space \cite{behrend},
and the maximal distance between $1$-symmetric spaces \cite{grundbacherkobos}.
In each of these examples, $\ell_1^2$ is part of every distance-maximizing pair.
A positive answer to Question~\ref{que:l1} would provide at least a partial explanation of this phenomenon,
as the maximization of the distance over any class of spaces containing $\ell_1^2$ would reduce to finding the space(s) farthest from $\ell_1^2$.
This would often be a significantly easier task.
For example, Stromquist established that the diameter of $\BM^2_s$ equals $\frac{3}{2}$,
and that the unique pair realizing it consists of $\ell_1^2$ and the space whose unit ball is the regular hexagon.
However, the proof of Stromquist is quite long and complicated.
In comparison, it was determined $20$ years earlier by Asplund \cite{asplund} (with a considerably simpler approach)
that the maximal distance to $\ell_1^2$ equals $\frac{3}{2}$.

Second, it is not difficult to prove that the triangle inequality
\[
    \dbm(X_1,X_2)
    \leq \dbm(X_1, \ell_1^2) \cdot \dbm(X_2, \ell_1^2)
\]
is always strict if $X_1$ and $X_2$ are non-isometric to $\ell_1^2$.
This means that a counterexample to \eqref{eq:quest_ineq} cannot be found by constructing a geodesic in $\BM^2_s$ containing the space $\ell_1^2$ in its interior.
Speaking informally, $\ell_1^2$ appears to be in a ``corner'' of the compactum $\BM^2_s$.
If this ``corner'' is narrow enough, the inequality \eqref{eq:quest_ineq} becomes plausible.
However, a positive answer to Open Question~\ref{que:l1} could quite possibly require a significant amount of effort.
\end{remark}

We conclude the paper with the proofs of Corollaries~\ref{cor:sym_embedd}~and~\ref{cor:eq_sets}.

\begin{proof}[Proof of Corollary~\ref{cor:sym_embedd}]
We need to prove that if $X_1, X_2$ are $2$-dimensional normed spaces such that $\dbm(X_i, X) \leq \dbm(X, \ell_1^2)^{\frac{1}{3}}$ for $i=1,2$, then 
\[
    \dbm(X_1 \oplus_1 \ell_1^{n-2}, X_2 \oplus_1 \ell_1^{n-2})
    = \dbm(X_1, X_2)
    =: d.
\]
This follows from Theorem~\ref{thm:sym_cones},
provided that $\dbm(X_1, \ell_1^2) \geq d$.
By the assumption and the triangle inequality, we have
\begin{align*}
    \dbm(X_1,\ell_1^2)
    & \geq \frac{\dbm(X,\ell_1^2)}{\dbm(X,X_1)}
    \geq \frac{\dbm(X,\ell_1^2)}{\dbm(X,\ell_1^2)^\frac{1}{3}}
    = \dbm(X,\ell_1^2)^{\frac{2}{3}}
    \\
    & \geq \dbm(X_1,X) \cdot \dbm(X,X_2)
    \geq \dbm(X_1,X_2),
\end{align*}
as required.
\end{proof}

\begin{proof}[Proof of Corollary~\ref{cor:eq_sets}]
By \cite[Theorem~$1$]{kobosswanepoel},
for every sufficiently large integer $N$ there exist $k := \lceil C^N \rceil$ (with $C > 1$ explicit) $2$-dimensional real normed spaces $X_1, \ldots, X_k$
such that $\dbm(X_i, X_j) = d_N^2$ for all $i \neq j$,
where $d_N= \frac{1}{\cos(\frac{\pi}{4N})}$.
Moreover, the unit balls of the spaces $X_i$ constructed in the proof are combinations of the regular $4N$-gon and the Euclidean disc, which satisfy $\dbm(X_i, \ell_2^2) \leq d_N$ for all $i=1, \ldots, k$.
Therefore, choosing $X = \ell_2^2$ in Corollary~\ref{cor:sym_embedd} gives for $n \geq 3$ and sufficiently large $N$ with $d_N \leq \dbm(\ell_2^2, \ell_1^2)^{\frac{1}{3}}=2^{\frac{1}{6}}$ that
\[
    \dbm(X_i \oplus_1 \ell_1^{n-2}, X_i \oplus_1 \ell_1^{n-2})
    = d_N^2
\]
for all $i \neq j$.
We conclude that the spaces $X_1 \oplus_1 \ell_1^{n-2}, \ldots, X_k \oplus_1 \ell_1^{n-2}$
form an equilateral set in $\BM_s^n$.
Since $k$ can be arbitrarily large, this concludes the proof.
\end{proof}

In closing the paper, let us point out another possible direction for future research.
The order of the size of the $d_N^2$-equilateral sets in $\BM_s^2$ constructed in \cite{kobosswanepoel} is essentially optimal,
as this order matches that of the largest $d_N^2$-separated set in $\BM_s^2$.
The asymptotic growth rate of the maximal cardinality of $(1+\varepsilon)$-separated sets in $\BM_s^n$ as $\varepsilon \to 0^+$
was determined by Bronstein \cite{bronstein} for all $n \geq 1$ and equals $\exp \left( \varepsilon^{\frac{1-n}{2}} \right)$ (see also \cite[Remark~$5$]{kobosswanepoel}).
While it might be possible that for all $n \geq 2$ the symmetric compactum $\BM^n_s$ contains equilateral sets with size of the corresponding order,
lifting the planar construction to higher dimensions cannot achieve this.
The process preserves the cardinality order of $\exp(\varepsilon^{-\frac{1}{2}})$.
Therefore, constructing equilateral sets in $\BM_s^n$ of sizes matching that of largest separated sets
requires a different approach for $n \geq 3$.

\section{Acknowledgements}

The research cooperation was funded by the program Excellence Initiative -- Research University at the Jagiellonian University in Kraków.

\bigskip

\textsc{Department of Mathematics, Technical University of Munich, Germany} \\
\textit{E-mail address}: \textbf{florian.grundbacher@tum.de}.

\vskip 0.2in

\textsc{Faculty of Mathematics and Computer Science, Jagiellonian University in Cracow, Poland} \\
\textit{E-mail address}: \textbf{tomasz.kobos@uj.edu.pl}

\vfill\eject


\begin{thebibliography}{99}

\bibitem{ader} O.~B.~Ader, \emph{An affine invariant of convex regions}, Duke Math.~J.~\textbf{4}(2) (1938), 291--299·

\bibitem{asplund} E.~Asplund, \emph{Comparison between plane symmetric convex bodies and parallelograms}, Math.~Scand.~\textbf{8} (1960), 171--180.

\bibitem{behrend} F.~Behrend, \emph{\"{U}ber einige Affininvarianten konvexer Bereiche}, Math.~Ann.~\textbf{113} (1937), 713–747 (in German).

\bibitem{bourgainszarek} J.~Bourgain, S.~J.~Szarek, \emph{The Banach--Mazur distance to the cube and the Dvoretzky-Rogers factorization}, Israel J.~Math.~\textbf{62}(2) (1988) 169--180.

\bibitem{bourgain} J.~Bourgain, L.~Tzafriri, \emph{Complements of subspaces of $\ell_p^n$; $p \geq 1$ which are uniquely determined}, In: J.~Lindenstrauss, V.~D.~Milman (eds.), Geometrical Aspects of Functional Analysis: Israel Seminar, 1985-86, Lecture Notes in Math., vol.~1267, Springer, Berlin, Heidelberg, 1987, 39--52. 

\bibitem{bronstein} E.~M.~Bronstein, \emph{$\varepsilon$-entropy of affine-equivalent convex bodies and Minkowski’s compactum}, Optimizatsiya \textbf{22}(39) (1978), 5--11 (in Russian).

\bibitem{glmp} Y.~Gordon, A.~E.~Litvak, M.~Meyer, A.~Pajor, \emph{John's decomposition in the general case and applications}, J.~Differential Geom.~\textbf{68}(1) (2004), 99--119.

\bibitem{grundbacherkobos} F.~Grundbacher, T.~Kobos, \emph{On certain extremal Banach--Mazur distances and Ader's characterization of distance ellipsoids}, Mathematika \textbf{72}(1) (2026), e70062.

\bibitem{grundbacherkobos2} F.~Grundbacher, T.~Kobos, \emph{John-type decompositions for affinely-optimal positions of convex bodies}, arXiv:\textbf{2602.14847} (2026).

\bibitem{hanner} O.~Hanner, \emph{Intersections of translates of convex bodies}, Math.~Scand.~\textbf{4} (1956), 65--87.

\bibitem{johnson} W.~B.~Johnson, G.~Schechtman, \emph{On the distance of subspaces of $l_p^n$ to $l_p^k$}, Trans.~Amer. Math.~Soc.~\textbf{324}(1) (1991), 319--329.

\bibitem{johnson2} W.~B.~Johnson, A.~Szankowski, \emph{The trace formula in Banach spaces}, Israel J.~Math.~\textbf{203} (2014), 389--404.

\bibitem{khrabrov} A.~I.~Khrabrov, \emph{Comparison of some distances between sums of $\ell_p^n$ spaces}, Funct.~Anal. Appl.~\textbf{38}(1) (2004), 75--77.

\bibitem{khrabrov2} A.~I.~Khrabrov, \emph{Estimates of distances between sums of the spaces $\ell_p^n$. II}, J.~Math.~Sci. \textbf{129}(4) (2005), 4040--4048.

\bibitem{kim} J.~Kim, \emph{Minimal volume product near Hanner polytopes}, J.~Funct.~Anal.~\textbf{266}(4) (2014), 2360--2402.

\bibitem{kobosswanepoel} T.~Kobos, K.~Swanepoel, \emph{Equilateral dimension of the planar Banach--Mazur compactum}, Proc.~Amer.~Math.~Soc.~\textbf{153}(10) (2025), 4423--4436.

\bibitem{kobosvarivoda} T.~Kobos, M.~Varivoda, \emph{On the Banach--Mazur distance in small dimensions}, Discrete Comput.~Geom.~\textbf{74}(2) (2025), 399--427.

\bibitem{plebanek} M.~Korpalski, G.~Plebanek, \emph{How many miles from $L_{\infty}$ to $\ell_{\infty}$?}, arXiv:\textbf{2511.12672v2} (2025).

\bibitem{piasecki} Ł.~Piasecki, \emph{On Banach space properties that are not invariant under the Banach--Mazur distance $1$}, J.~Math.~Anal.~Appl.~\textbf{467}(2) (2018), 1129--1147.

\bibitem{rudelson} M.~Rudelson, \emph{Estimates of the weak distance between finite-dimensional Banach spaces}, Israel J.~Math.~\textbf{89} (1995), 189--204.

\bibitem{stromquist} W.~Stromquist, \emph{The maximum distance between two-dimensional Banach spaces}, Math.~Scand.~\textbf{48} (1981), 205--225

\bibitem{tomczaksymm} N.~Tomczak-Jaegermann, \emph{The Banach--Mazur distance between symmetric spaces},  Israel J.~Math.~\textbf{46}(1-2) (1983), 40--66.

\end{thebibliography}
\end{document}